\documentclass[12pt]{article}

\usepackage{geometry}
\usepackage{bold-extra}

\usepackage{hyperref}
\usepackage{soul}
\usepackage{amsthm}

\usepackage{mathtools}
\usepackage{psfrag}
\usepackage{amssymb}
\usepackage{amsmath}
\usepackage{graphicx}
\usepackage{xcolor}
\allowdisplaybreaks
\usepackage[color=green!40]{todonotes}
\usepackage{enumitem}
\usepackage{pifont}

\usepackage{bbding}

\setitemize{label=\scriptsize{\ding{118}}}
\usepackage{framed}

\newtheorem{theorem}{Theorem}[section]
\newtheorem{lemma}[theorem]{Lemma}
\newtheorem{corollary}[theorem]{Corollary}

\newtheorem{alg}{Algorithm}

\theoremstyle{definition}
\newtheorem{definition}[theorem]{Definition}

\newcommand{\RR}{\mathbb{R}}
\newcommand{\ZZ}{\mathbb{Z}}
\newcommand{\NN}{\mathbb{N}}
\newcommand{\sph}{\mathbb{S}}
\newcommand{\ph}{\varphi}

\newcommand{\xx}{x}
\newcommand{\ww}{w}
\newcommand{\yy}{y}

\newcommand{\vv}{v}
\newcommand{\uu}{u}

\newcommand{\al}{\alpha}
\newcommand{\om}{\omega}
\newcommand{\psibar}{\psi'}

\newcommand{\Fo}{\mathcal{F}}
\newcommand{\Co}{\mathcal{C}}
\newcommand{\Vo}{\mathcal{V}}
\newcommand{\Ho}{\mathcal H}
\newcommand{\Ro}{\mathcal R}
\newcommand{\Xo}{\mathcal X}

\newcommand{\F}{F}

\newcommand{\f}{f}

\newcommand{\FF}{\textsc F}
\newcommand{\GG}{\textsc G}
\newcommand{\HH}{\textsc H}

\newcommand{\axis}[2]{a(#1,#2)}
\newcommand{\abs}[1]{\left\lvert#1\right\rvert}
\newcommand{\sabs}[1]{\lvert#1\rvert}
\newcommand{\norm}[1]{\left\lVert#1\right\rVert}
\newcommand{\snorm}[1]{\lVert#1\rVert}
\newcommand{\kl}[1]{\left(#1\right)}
\newcommand{\skl}[1]{(#1)}

\newcommand{\set}[1]{\left\{#1\right\}}
\newcommand{\sset}[1]{\{#1\}}

\newcommand{\edot}{\,\cdot\,}
\newcommand{\rmd}[1]{\mathrm d#1}
\newcommand{\dS}{\mathrm{d}S}
\newcommand{\dr}{\mathrm{d}r}
\newcommand{\dbeta}{\mathrm{d}\beta}
\newcommand{\ds}{\mathrm{d}s}
\newcommand{\dy}{\mathrm{d}\yy}
\newcommand{\arccosh}{\operatorname{arccosh}}
\newcommand{\suppp}{\operatorname{supp}}

\makeatletter
\newcommand*\bigcdot{\mathpalette\bigcdot@{.6}}
\newcommand*\bigcdot@[2]{\mathbin{\vcenter{\hbox{\scalebox{#2}{$\m@th#1\bullet$}}}}}
\makeatother

\newcommand\inner[2]{{#1}\bigcdot {#2}}

\numberwithin{theorem}{section}
\numberwithin{figure}{section}
\numberwithin{equation}{section}

\title{Analytic inversion of a conical Radon transform arising in application of Compton cameras on the cylinder}

\author{Sunghwan Moon\footnotemark[2] \and Markus Haltmeier\footnotemark[3]}

\date{\small \footnotemark[2] Department of Mathematical Sciences, Ulsan National Institute of Science and Technology\\
Ulsan 44919, Republic of Korea.\\
{\tt shmoon@unist.ac.kr}\\[1em]
\footnotemark[3] Department of Mathematics, University of Innsbruck\\
Technikerstrasse 13, A-6020 Innsbruck, Austria\\
 {\tt Markus.Haltmeier@uibk.ac.at}}

\begin{document}
\maketitle
\newcommand{\slugmaster}{%
\slugger{siims}{xxxx}{xx}{x}{x--x}}%slugger should be set to juq, siads, sifin, or siims

\begin{abstract}
Single photon emission computed tomography (SPECT) is a well established clinical tool for   functional imaging.   A limitation of current SPECT systems is the use of mechanical collimation, where only a small fraction of the  emitted photons is actually used for image reconstruction. This results in large noise level and finally in a limited spatial  resolution. In order to decrease the noise level and to increase the imaging resolution, Compton  cameras have been proposed  as an alternative  to mechanical collimators. Image reconstruction in SPECT with  Compton cameras yields to the problem of recovering a marker distribution from integrals over conical surfaces. Due to this and other applications, such conical Radon transforms recently got significant  attention. In the current  paper  we consider the case where the cones of integration  have vertices on a circular cylinder and axis pointing to the symmetry axis of the cylinder. As main results we derive analytic  reconstruction methods for the considered  transform. We also investigate the V-line transform with vertices on a circle and symmetry axis orthogonal to the circle, which arises in the special  case where the absorber distribution is located in a horizontal plane.
 


\end{abstract}

\section{Introduction}
\label{sec:intro}

In this  paper we study the  inversion of a conical Radon transform that maps a function defined in three dimensional space  to its  integrals over a special family of cones. Recovering a function from   integrals over cones arises in SPECT using Compton cameras. These type measurement devices have been introduced as an alternative  to classical gamma cameras based on mechanical collimators with  increased sensitivity \cite{EveFleTidNig77,Sin83,TodNigEve74}. Inversion of  conical Radon transforms is also relevant for  single scattering optical 
tomography~\cite{FloSchMar09} or  Compton  scattering  imaging~\cite{MorEtAl10}. Recently various versions the  conical Radon transforms have been studied (see for example~\cite{Allmaras13,BasZenGul98,CreBon94,GouAmb13,Hal14a,JunMoo15,Moon16a,MorEtAl10,schiefeneder2016radon,Smi05,smith2011line,Terzioglu15} and the references therein). The instance of the conical Radon transform that we study in this paper arises in application of SPECT using a cylindrical  Compton camera.

\subsection{Compton cameras in SPECT}
\label{sec:compton}

In SPECT, weakly radioactive tracers are given to the patient and participate in physiological processes.  The radioactive tracers can be detected through the  emission of gamma ray photons allowing  to infer information about physiological processes.   In order to obtain sufficient location information about  the emitted photons,  the standard  approach in SPECT is to use collimators, which only record photons that enter the detector surface vertically.  As illustrated in the left image in Figure~\ref{fig:spect},  such data provide  integrals of the tracer distribution over straight lines, and reconstructing the tracer distribution can be performed by inverting the  (attenuated) ray  transform~\cite{wernick2004emission}. A main drawback of the use of mechanical collimators is that they remove most photons, and therefore the number  of   recorded photons  is low.  To increase the  sensitivity, the use of Compton cameras has been proposed in~\cite{EveFleTidNig77,Sin83,TodNigEve74}. Opposed to classical gamma cameras,
Compton cameras use two detector arrays  in order to avoid  the use of a mechanical collimator.

\psfrag{C}{$C$}
\psfrag{D}{$D$}
\psfrag{S}{$D_s$}
\psfrag{A}{$D_a$}
\psfrag{a}{$x_s$}
\psfrag{b}{$x_a$}
\psfrag{B}{axis}
\psfrag{w}{$\psi$}
\begin{figure}[tbh!]\centering
\includegraphics[width =0.4\textwidth]{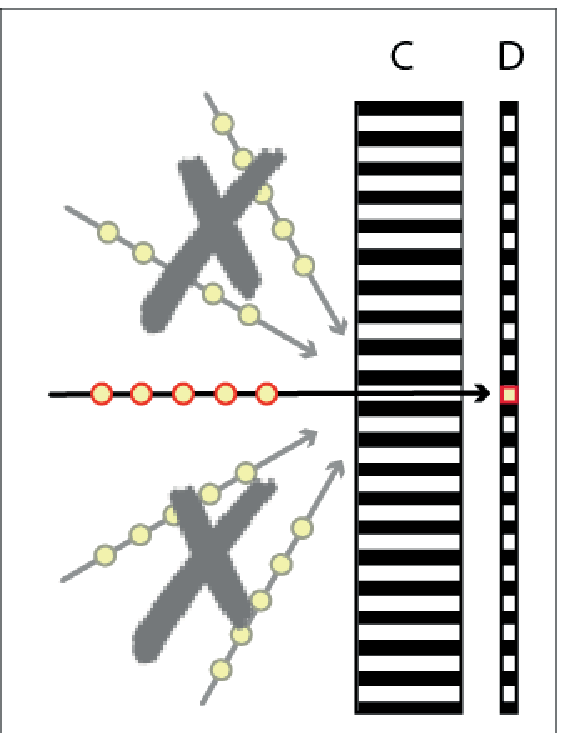} \hspace{0.1\textwidth}
\includegraphics[width =0.4\textwidth]{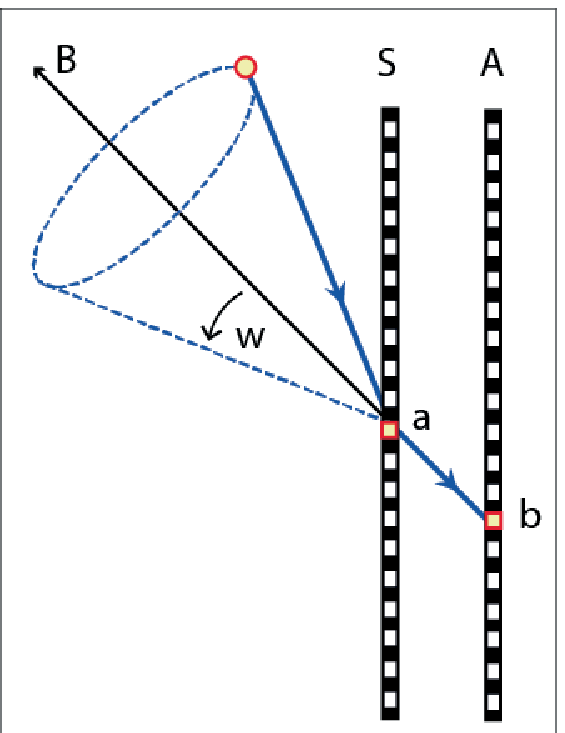}
\caption{{\scshape Collimator versus Compton camera.}  Left: In standard gamma  cameras, a collimator $C$ is  inserted, which only observes  photons propagating orthogonal to the detector plane. The location of any  emitted photon  can be traced back to a straight line.  Right: A Compton camera consists of  two detector arrays $D_s$ and $D_a$. Every observed photon can be traced back to the surface of a cone.\label{fig:spect}}
\end{figure}

As illustrated in  the right  image in Figure~\ref{fig:spect}, a Compton camera  consists of a scatter detector  array $D_s$ and an absorption detector array $D_a$.  A photon emitted in the direction of the camera undergoes Compton scattering in $D_s$, and is absorbed in $D_a$.
In each detector, the position and the energy of the photon are 
measured~\cite{Sin83}. The measured energies can be used to determine the  scattering angle  or half opening angle $\psi$ at $D_s$ via the Compton scattering formula  $\cos(\psi)=1-m c^2 (E_s-E_a)/(E_sE_a)$, where $m$ is the electron mass, $c$ the speed of light, $E_s$ the photon energy at $D_s$, and $E_a$ the energy of the photon measured at $D_a$.
Using this information, one  can conclude  that the observed photon must have been emitted on the surface of a circular cone, where the vertex is given by the position $x_s$ at  $D_s$, the central axis  points from $x_a$ (the position at  $D_a$) to $x_s$,  and the scattering angle  is given by $\psi$.

Now suppose we have given a distribution of traces $\f \colon \RR^3 \to \RR$ which emit photons uniformly in all directions. The expected number of photons  recorded  with  data $(x_s,E_s)$ and $(x_a,E_a)$   is proportional to $K(\psi) I$, where  $I$ is the integral over the cone determined by  $(x_s,E_s)$ and $(x_a,E_a)$ and $ K(\psi)$ is the Klein-Nishina distribution which describes  the probability that a given photon scatters by angle $\psi$. The Klein-Nishina distribution is known explicitly and well bounded from  below for typical photon energies. After rescaling we can therefore assume that Compton cameras provide noisy versions of  $I$, from which the tracer distribution $\f$ has to be recovered.

\psfrag{w}{\scriptsize $\psi$}
\psfrag{v}{\scriptsize $\beta$}
\psfrag{A}{(a)}
\psfrag{B}{(b)}
\psfrag{C}{(c)}
\psfrag{a}{\scriptsize $a$}
\psfrag{f}{\scriptsize $f$}
\psfrag{e}{\scriptsize $x_3$}
\psfrag{Z}{\scriptsize $\sph^1 \times \RR$}
\psfrag{S}{\scriptsize $\sph^1$}
\psfrag{L}{\scriptsize $\RR$}
\begin{figure}
\begin{center}
  \includegraphics[width =\textwidth]{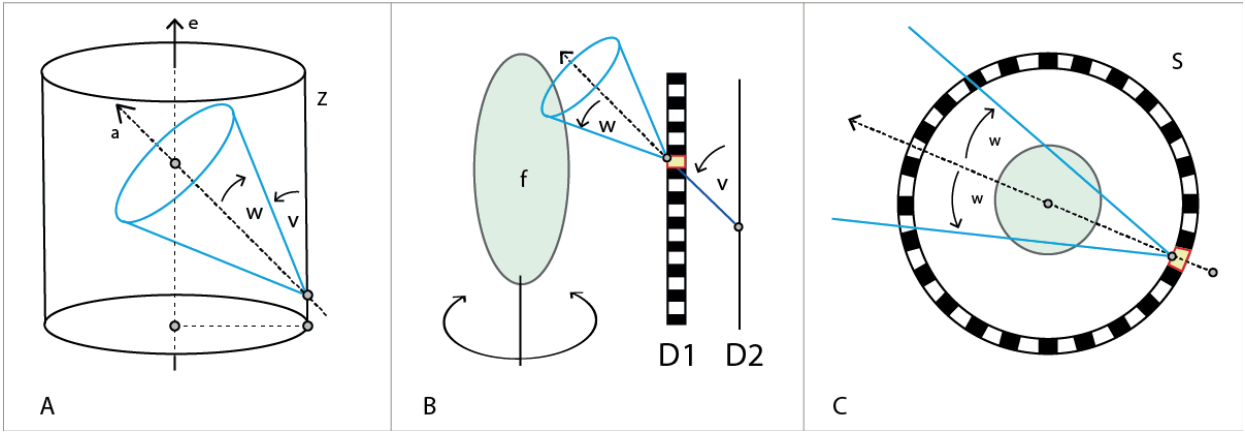}
\end{center}
  \caption{{\scshape Conical Radon transform on the cylinder.}
  (a) The conical Radon transform integrates a function  over all cones with the vertex on  $\sph^1 \times \RR$  and the axis intersecting $\set{(0,0)}
  \times \RR$. (b) Vertical cross section. The conical integrals may be obtained by rotating a one-dimensional Compton camera rotated around  the $e_3$-axis. (c) Horizontal cross section. The conical integrals reduce to  integrals over V-shaped, with the vertex on $\sph^1$ and the symmetry axis pointing  to the origin.}\label{fig:cylinder}
\end{figure}

\subsection{The conical Radon transform on the cylinder}

In this work we consider the situation where the scattering detector  forms a cylindrical surface $\sph^1\times\RR$ and that the  source  distribution  $\f$ is assumed to be supported  inside this cylinder.  We  assume that the available data consist of   integrals over  all cones with  the vertex on $\sph^1\times\RR$ and the central axis pointing to  $  \set{(0,0)} \times \RR$; see Figure~\ref{fig:cylinder} (a).
Our goal is to recover $\f$ from its conical Radon transform, consisting of  integrals of $\f$  over these  cones.
One possible way  to realize a cylindrical Compton camera is illustrated in Figure~\ref{fig:cylinder}~(b), where, in a first step,
data are collected with a one-dimensional Compton
camera \cite{BasZenGul97,JunMoo16}.  Such a one-dimensional Compton camera consists of two  linear detector arrays
and records integrals over cones with axis intersecting both linear detectors.
In order to obtain the considered data, the one-dimensional Compton camera  is rotated around the $x_3$-axis.  In~\cite{JunMoo16} it has been  shown that data of a one-dimensional Compton camera without rotation is theoretically  (almost) sufficient to recover the absorber distribution.  While $\f$ has 3-dimensional domain, the   class of cones  we consider depends on four variables, and so we are facing an overdetermined inverse problem.  Due to low photon counts  the use of such over-determinated data is actually a desired feature of Compton cameras (see, for example, the discussion in~\cite{Allmaras13}).

Additionally, we consider a special two-dimensional version of the conical Radon transform. Suppose that the support of the phantom is thin and contained in a horizontal plane, such that we can model  it by $f(x_1,x_2, x_3) = \delta(x_3) \F(x_1, x_2)$, where $\delta$ is the one-dimensional delta-distribution.
In such  a situation it is reasonable  to restrict the set of vertices to  $ \sph^1 \times \set{0}$. The intersection of any
corresponding cone  with the plane $\set{x \in \RR^3 \mid x_3 = 0}$ becomes a V-line with the vertex on $\sph^1$  and the axis pointing to the origin; see Figure~\ref{fig:cylinder}~(c).
Reconstructing the function $\F$ from the resulting V-line transform will be considered in Section~\ref{sec:V-line}. In fact the developed  inversion method for the V-line transform is the basis of one of our inversion  methods for the three-dimensional conical Radon transform.
Another interesting two-dimensional case would arise when $\f$  is restricted to a vertical plane and  the vertices of the cones are located on a line pointing in the $x_3$-direction. The resulting V-line transform with vertices on a  line been studied in previous works, such as \cite{Allmaras13,BasZenGul97,JunMoo15,Hal14a,TruNgu11}, and therefore will  not be investigated in the present work.

\subsection{Outline}

This manuscript is organized as follows.
In Section~\ref{sec:V-line} we define the V-line transform with vertices
on the circle and derive an explicit inversion formula based on a Fourier series expansion. This in particular implies the invertibility of the V-line transform. In Section~\ref{sec:cone} we present two inversion methods for inverting the conical Radon transform with vertices on the
cylinder and the axis pointing to the symmetry axis of the cylinder. The first method even works when a  radial weight included in the conical Radon transform and is based on reducing the  conical Radon transform to the V-line transform. The second method uses an approach of Smith~\cite{Smi05} and reduces  the conical Radon transform to the standard Radon transform. This further allows us to show stability of inverting the conical Radon transform. Generalizations of our results to higher dimension are presented in the appendix.

\section{V-line transform with vertices on a circle}
\label{sec:V-line}

In this section  we study the the V-line transform on the circle.
We derive an explicit inversion formula using a Fourier series  expansion,
which in particular implies the invertibility of the V-line transform, and
further derive a numerical reconstruction algorithm  based
on our inversion formula.

\subsection{Explicit inversion formula}

For   $\ph \in [0, 2\pi) $ we  write $\theta(\ph) \coloneqq (\cos\ph,\sin\ph)$. Further, we denote by $D_1(0)\coloneqq \sset{x\in \RR^2 \mid \snorm{x} <1 }$ the unit disc in $\RR^2$ and by $C_c^\infty(D_1(0))$ the set  of all $C^\infty$-functions $\F \colon \RR^2 \to \RR$ with $\suppp (\F) \subseteq D_1(0)$.

\begin{definition}[The V-line transform on the circle]
Let $\F \in C_c^\infty (D_1(0))$. We define  the \ul{V-line transform}  of $F$ by
\begin{equation}  \label{eq:V-line transform}
\Vo \F  \colon [0, 2\pi)  \times (0, \pi/2)  \to \RR \colon
  (\ph, \psi) \mapsto   \sum_{\sigma = \pm 1}\int_0^\infty \F(\theta(\ph)-r\theta(\ph - \sigma\psi)  )  \, \dr \,.
\end{equation}
\end{definition}

The V-line transform integrates the function $F$ over V-lines (one-sided cones in the plane), having the vertex
$\theta(\ph)  \coloneqq (\cos\ph,\sin\ph)\in \sph^1$,
symmetry axis  $\sset{-r\theta(\ph) \mid r >0}$ and half opening angle $\psi$.
In the following we will frequently make use of the 2-dimensional   (regular)  Radon transform of a function $\F \in C_c^\infty (\RR^2)$, defined by
\begin{equation*}
	(\Ro\F)(\al,s)
	\coloneqq  \int_{\RR}
	\F \kl{ s \cos(\al)-t\sin(\al), s\sin(\al) + t\cos(\al) } \rmd t
	\quad \text{ for }  (\al,s)\in [0,2\pi) \times \RR \,.
\end{equation*}
The Radon transform integrates  the function $\F$ over the line
$\sset{x \in \RR^2 \mid  \inner{(\cos(\al), \sin(\al))}{x} = s}$ having a normal vector $(\cos(\al), \sin(\al))$ and an oriented distance $s \in \RR$ from the origin.

The  inversion approach we present below  uses the Fourier series of $\F$
and $\Vo \F$ with respect the angular variables,
\begin{align} \label{eq:fn}
\F\kl{r\theta(\ph)}      & = \sum_{n \in \ZZ} \F_n(r)\,e^{in\ph}
&& \text{with} \quad \F_n(r) \coloneqq  \frac{1}{2\pi}\int^{2\pi}_0 \F(r\theta(\ph))\,e^{-in\ph}\rmd\ph \,,
\\  \label{eq:gn}
(\Vo\F)(\ph,\psi)   &=  \sum_{n \in \ZZ} (\Vo\F)_n(\psi)\,e^{in\ph}
&& \text{with} \quad (\Vo\F)_n(\psi) \coloneqq \frac{1}{2\pi}\int^{2\pi}_0 (\Vo\F)(\ph,\psi)\,e^{-in\ph}\rmd{\ph}\,.
 \end{align}
For $k \geq 0$,  we denote by $T_k(z)$  and $U_k(z)$  the Chebyshev polynomials of the first and second kind, respectively,
where
 \begin{align*}
T_k(z)
&\coloneqq \begin{cases}
\cos \kl{k \arccos(z)} & \text{ for } \abs{z} \leq 1 \\
\cosh (k \operatorname{arccosh}(z)) & \text{ for } z >  1\\
(-1)^k\cosh (k \operatorname{arccosh}(-z)) & \text{ for } z <  -1
\end{cases}
\\
U_k(z)
&\coloneqq \begin{cases}
\sin ((k+1)\arccos(z))/\sin (\arccos(z)) & \text{ for } \abs{z} \leq 1 \\
\sinh ((k+1)\operatorname{arccosh}(z))/\sinh (\operatorname{arccosh}(z)) & \text{ for } z >  1 \\
(-1)^k\sinh ((k+1)\operatorname{arccosh}(-z))/\sinh (\operatorname{arccosh}(-z))  & \text{ for } z <  -1\,,
\end{cases}
\end{align*}
and  set $ U_{-1} \coloneqq 0$.

Our strategy for inverting $\Vo$ is to recover the  Fourier coefficient $\F_n$ from $(\Vo\F)_n$ for any $n \in \ZZ$. For that purpose we proceed by  setting up a one-dimensional integral  equation for $\F_n$  in terms  of $(\Ro\F)_n$ that will subsequently be solved explicitly.

\begin{lemma}[Expressing\label{lem:vline}  $(\Vo\F)_n$ in terms of $\F_n$]
Suppose $\F \in C_c^\infty(D_1(0))$ and let
 $\F_n$ and $(\Vo\F)_n$   denote the Fourier coefficients of $\F$ and
 $\Vo \F$ as defined in \eqref{eq:fn} and \eqref{eq:gn}, respectively.
 Then,  for all  $(n, \psi) \in \ZZ \times (0,\pi/2)$, we have
\begin{equation} \label{eq:vline}
(\Vo\F)_n(\psi) =   4 \cos(n(\psi - \pi/2) ) \int_{\sin(\psi)}^1 \F_n(r) \frac{T_{|n|}(\sin(\psi)/r)}{\sqrt{r^2-\sin^2(\psi)}}  \rmd r  \,.
\end{equation}
\end{lemma}

\begin{proof} \mbox{}
From the definitions of $\Vo \F$ and $\Ro \F$  we have the following  relation
\begin{equation} \label{eq:VRT}
	\Vo \F (\ph, \psi) =
	\Ro \F(\ph - \psi   +\pi/2, \sin(\psi))
	+
	\Ro \F(\ph+ \psi   -\pi/2,   \sin(\psi))
	\,.
\end{equation}
Now let $(\Ro\F)_n(s) \coloneqq \frac{1}{2\pi}\int^{2\pi}_0 (\Ro\F)(\al,s)\,e^{-in\al}  \rmd\al$ denote the $n$-th Fourier coefficient  of $\Ro \F$ with respect to the angular variable.
Equation \eqref{eq:VRT}, the definition of the Fourier  coefficients  of $\Vo \F$ and $\Ro \F$, and two variable substitutions yield
\begin{multline}\label{eq:vfnrf}
(\Vo\F)_n(\psi)
\\
\begin{aligned}
&=\frac{1}{2\pi}\int^{2\pi}_0
\left[
\Ro \F(\ph - \psi    +\pi/2,   \sin(\psi) )
+
\Ro \F(\ph + \psi   -\pi/2,   \sin(\psi))
\right]
 \,e^{-in\ph}\rmd{\ph}
 \\ &=
\frac{1}{2\pi}\int^{2\pi}_0
\Ro \F(\al,\sin(\psi))\, e^{-in(\al+\psi - \pi/2)} \rmd\al
+
\frac{1}{2\pi}\int^{2\pi}_0
\Ro \F(\al,\sin(\psi))\,e^{-in(\al-\psi + \pi/2)}\rmd\al\\
&=
(\Ro \F)_n (\sin(\psi))\,e^{-in(\psi - \pi/2)}+(\Ro \F)_n(\sin(\psi))\,e^{ in(\psi - \pi/2)}
\\
&= 2    \cos(n(\psi - \pi/2) ) \,(\Ro \F)_n(\sin(\psi))
\,.
\end{aligned}
\end{multline}
Next we note the relation $(\Ro \F)_n(s) =   2 \int_{s}^1 \F_n(r) \frac{T_{|n|}(s/r)}{\sqrt{r^2-s^2}}  \rmd r$, that  was first derived by Cormack in~\cite{Cor63}. Combining this with the last displayed equation yields  \eqref{eq:vline}.
 \end{proof}

Lemma~\ref{lem:vline} together  with known inversion formulas for the Radon transform yields  the following explicit inversion formulas for  the V-line transform on the circle.

\begin{theorem}[Inversion\label{thm:inv2d} of the V-line transform on the circle]
Suppose $\F \in C^\infty(D_1(0))$ and let
 $\F_n$ and  $(\Vo\F)_n$  denote the Fourier coefficients of $\F$ and  $\Vo \F$,
 as defined in \eqref{eq:fn} and \eqref{eq:gn}.
 Then the  following inversion formulas hold:
\begin{align}\label{eq:inv2a}
\F_n(r)  &=  -\frac{1}{2\pi} \int_r^1 \frac{\partial }{\partial s}
\left[  \frac{(\Vo\F)_n(\arcsin(s))}{\cos(n(\arcsin(s) - \pi/2) )} \right]  \frac{T_{|n|}(s/r)}{\sqrt{s^2-r^2}}  \,  \rmd s
\\\label{eq:inv2b}
\F_n(r) &=  -\frac{1}{2\pi r} \biggl\{
\int_r^1 \frac{\partial }{\partial s}
\left[  \frac{(\Vo\F)_n(\arcsin(s))}{\cos(n(\arcsin(s) - \pi/2) )} \right]
\frac{\left[ s/r + \sqrt{s^2/r^2-1}\right]^{-\abs{n}}}{\sqrt{s^2/r^2-1}} \, \rmd s
\\ & \hspace{0.2\textwidth} \nonumber
-\int_0^r \frac{\partial }{\partial s}
\left[  \frac{(\Vo\F)_n(\arcsin(s))}{\cos(n( \arcsin(s)-\pi/2))} \right]
U_{\abs{n}-1}(s/r)  \, \rmd s
\biggr\} \,.\end{align}
\end{theorem}

\begin{proof} Let $(\Ro\F)_n$ denote the $n$-th Fourier coefficient of $\Ro \F$. Then   \eqref{eq:vfnrf} implies
\begin{equation} \label{eq:inv-aux}
	(\Ro\F)_n(s) = \frac{(\Vo\F)_n(\arcsin (s))}{2 \cos(n(\arcsin(s) - \pi/2) )}  \,.
\end{equation}
Next we recall the following inversion formulas for the Radon transform:
\begin{align}\label{eq:invRa}
\F_n(r)  &=  -\frac{1}{\pi} \int_r^1 \frac{\partial (\Ro\F)_n(s)}{\partial s}
 \frac{T_{|n|}(s/r)}{\sqrt{s^2-r^2}}  \,  \rmd s
\\ \label{eq:invRb}
\F_n(r) &=  -\frac{1}{\pi r} \biggl\{
\int_r^1 \frac{\partial (\Ro\F)_n(s)}{\partial s}
\frac{\left[ s/r + \sqrt{s^2/r^2-1}\right]^{-\sabs{n}}}{\sqrt{s^2/r^2-1}} \, \rmd s
\\ \nonumber & \hspace{0.25\textwidth}
-\int_0^r \frac{\partial (\Ro\F)_n(s)}{\partial s}
U_{\abs{n}-1}(s/r)  \, \rmd s
\biggr\} \,.\end{align}
Here \eqref{eq:invRa} is Cormack's inversion formula~\cite{Cor63} and   \eqref{eq:invRb} an inversion formula with better stability properties first
derived in~\cite{Per75}.  Inserting \eqref{eq:inv-aux}  in  \eqref{eq:invRa} yields the inversion formula~\eqref{eq:inv2a} whereas  inserting   \eqref{eq:invRb}
yields~\eqref{eq:inv2b}.
\end{proof}

Theorem~\ref{thm:inv2d} in particular implies that the  V-line transform  is uniquely invertible.   Further, it implies the following inversion method for the V-line transform:
\begin{framed}
\begin{alg}[Inversion of the \label{alg:vrt} V-line transform]\mbox{}
\begin{itemize}[leftmargin=3em]
\item {\bfseries\scshape\ul{Step 1:}}
Compute  the Fourier  coefficients $(\Vo \F)_n$
of $\Vo \F$;  see~\eqref{eq:gn}.

\item {\scshape\bfseries\ul{Step 2:}}
Recover the Fourier  coefficients $\F_n$ from $(\Vo \F)_n$
by~\eqref{eq:inv2a} or \eqref{eq:inv2b}.

\item{\scshape\bfseries\ul{Step 3:}}
Recover $\F$ from its  Fourier  coefficients $\F_n$; see  \eqref{eq:fn}.
\end{itemize}
\end{alg}
\end{framed}

The inversion formula~\eqref{eq:inv2a}  solves the  exterior problem for the V-line transform, because for  reconstructing $\F_n(r)$ it only  uses integrals over V-lines that do not intersect the  disc $\sset{x \in \RR^2 \mid \abs{x} \leq r }$.  Such data can be stably obtained from exterior data of the Radon transform. As a consequence, evaluating~\eqref{eq:inv2a} is numerically  unstable (severely ill-posed). For the following, we
therefore only consider the inversion formula  \eqref{eq:inv2b} and we  demonstrate that  it  can be evaluated stably and efficiently.

\subsection{Numerical implementation}

In this subsection we describe how to numerically  implement Algorithm~\ref{alg:vrt}. In  our implementation, we discretize any step in Algorithm~\ref{alg:vrt}. For computing the  Fourier coefficients in Step~1 and for evaluating the Fourier series in Step~3, we use the standard FFT algorithm. For Step~1, the FFT algorithm outputs approximations   $\GG[n,j]  \simeq  (\Vo  \F )_n(\arcsin(s_j))$ for $n \in \set{-N/2, -N/2+1,\dots, N/2-1}$ and $j \in \set{0, \dots,  M}$.

The main issue in the reconstruction procedure is implementing  the inversion formula~\eqref{eq:inv2b} in Step~2. Consider equidistant grid points $r_i = i/M$ for $i \in \set{0, \dots, M}$ and rewrite  the inversion formula~\eqref{eq:inv2b}   in the form
\begin{align*}
\F_n(r) &=
 \frac{1}{\pi }  \left\{
 \int_0^r
g_n'(s)
U_{\abs{n}-1}(s/r)  \, \frac{\rmd s}{r}
-
\int_r^1
g_n'(s)
\frac{\bigl[ s/r + \sqrt{s^2/r^2-1} \, \bigr]^{-\abs{n}}}{\sqrt{s^2/r^2-1}} \, \frac{\rmd s}{r}
\right\} \,,
\\
g_n'(s)
& \coloneqq \frac{\partial}{ \partial s}\left[  \frac{(\Vo\F)_n(\arcsin(s))}{2 \cos(n(\arcsin(s) - \pi/2) )} \right] \,.
\end{align*}
These formulas are used  for finding approximations to $ \F_n(r_i)$ as follows.
First, for  some small regularization parameter $\epsilon >0$ (that accounts for  instabilities due to the zeros in the denominator in the definition to $g_n'$) we define
\begin{equation} \label{eq:tik}
	\HH[n,j] \coloneqq
	\frac{\cos(n(\arcsin(s_j) - \pi/2) ) \,
	\GG[n,j] }{\epsilon^2 + \cos(n(\arcsin(s_j) - \pi/2) )^2} \,,
\end{equation}
and take   $g_{n}'[j]  \coloneqq    (\HH[n,j+1]- \HH[n,j])/(2 M)$ as an approximation of $g_n'(s)$ on the interval $[r_j, r_{j+1}]$. Using such an approximation,
we  obtain
\begin{align*}
\F_n(r_i)
& =
 \frac{1}{\pi }  \left\{
\int_0^{r_i} g_n'(s)
U_{\abs{n}-1}(s/r_i)  \, \frac{\rmd s}{r_i}
-
\int_{r_i}^1
g_n'(s)
\frac{\bigl[ s/r_i + \sqrt{s^2/r_i^2-1} \, \bigr]^{-\abs{n}}}{\sqrt{s^2/r_i^2-1}} \, \frac{\rmd s}{r_i}
 \right\}
\\
&= \frac{1}{\pi }
  \Biggl\{  \sum_{j=0}^{i-1}
    \int_{r_j}^{r_{j+1}}
    g_n'(s)
    U_{\abs{n}-1}(s/r_i)  \, \frac{\rmd s}{r_i}
    \\
    & {}\hspace{0.2\textwidth}\qquad -
    \sum_{j=i}^{M-1}
  \int_{r_j}^{r_{j+1}}
      g_n'(s)
  \frac{\bigl[ s/r_i + \sqrt{s^2/r_i^2-1} \, \bigr]^{-\abs{n}}}{\sqrt{s^2/r_i^2-1}} \, \frac{\rmd s}{r_i}
    \Biggr\}
    \\
&\simeq \frac{1}{\pi }
  \Biggl\{  \sum_{j=0}^{i-1}
  g_{n}'[j]
    \int_{r_j}^{r_{j+1}}
    U_{\abs{n}-1}(s/r_i)  \, \frac{\rmd s}{r_i}
    \\
    & {}\hspace{0.2\textwidth}\qquad -
    \sum_{j=i}^{N-1}
    g_{n}'[j]
  \int_{r_j}^{r_{j+1}}
  \frac{\bigl[ s/r_i + \sqrt{s^2/r_i^2-1} \, \bigr]^{-\abs{n}}}{\sqrt{s^2/r_i^2-1}} \, \frac{\rmd s}{r_i}
    \Biggr\}\,.
\end{align*}
By elementary integration one verifies that
 \begin{align} \label{eq:w1}
\forall j \in \set{0, \dots, i-1} \colon \quad
w_{i,j}^{(n)}
&\coloneqq
\int_{r_j}^{r_{j+1}}
    U_{\abs{n}-1}(s/r_i)  \, \frac{\rmd s}{r_i}
 \\&= \nonumber
 \begin{cases}
  0   & \text{ for } n = 0
  \\
  \dfrac{1}{\abs{n}} \kl{ T_{\abs{n}}( r_{j+1}/r_i )
  - T_{\abs{n}}( r_j/r_i ) } & \text{ for } n \neq 0    \,,
  \end{cases}
 \end{align}
 and
\begin{multline}\label{eq:w2}
\forall j \in \set{i, \dots, M-1} \colon \quad
w_{i,j}^{(n)} \coloneqq
\int_{r_j}^{r_{j+1}}
  \frac{\bigl[ s/r_i + \sqrt{s^2/r_i^2-1} \, \bigr]^{-\abs{n}}}{\sqrt{s^2/r_i^2-1}} \, \frac{\rmd s}{r_i}
\\ =
 \begin{cases}
   \log\kl{r_{j+1} + \sqrt{r_{j+1}^2-r_i^2}\, }
  -  \log \kl{r_j + \sqrt{r_j^2-r_i^2}\, }  & \text{ for } n = 0
  \\
- \dfrac{1}{\abs{n}} \kl{e^{ - \abs{n}  \arccosh( r_{j+1}/r_i )}
  - e^{ - \abs{n}  \arccosh( r_j/r_i )}}  & \text{ for } n \neq 0  \,.
  \end{cases}
 \end{multline}
Consequently, we obtain 
\begin{equation} \label{eq:final}
	\forall (n,i) \in \set{-N/2, \dots, N/2-1} \times \set{1, \dots, M}
	\colon
	\quad
	\FF[n,i]
	= \sum_{j=0}^{M-1}
	w_{i,j}^{(n)}
	 \frac{\HH[n,j]- \HH[n,j]}{2 M} \,,
\end{equation}
where $\HH[n,j]$ is given by \eqref{eq:tik}, the weights $w_{i,j}^{(n)}$ are defined by \eqref{eq:w1} and \eqref{eq:w2}, and
$\FF[n,i]$ is the desired approximation to $\F_n(r_i)$.

\begin{figure}[htb!]
\includegraphics[width=0.32\textwidth]{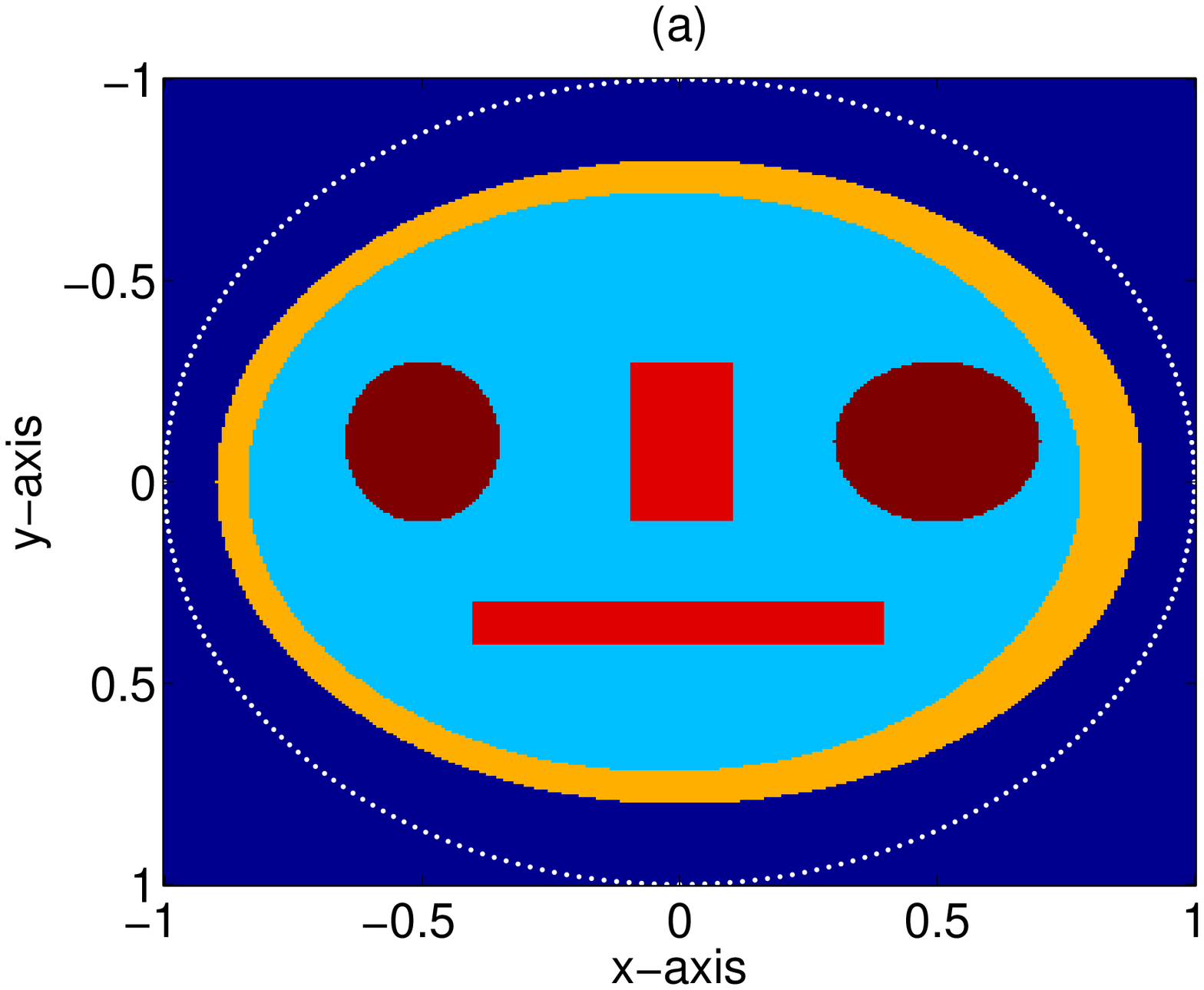}
\includegraphics[width=0.32\textwidth]{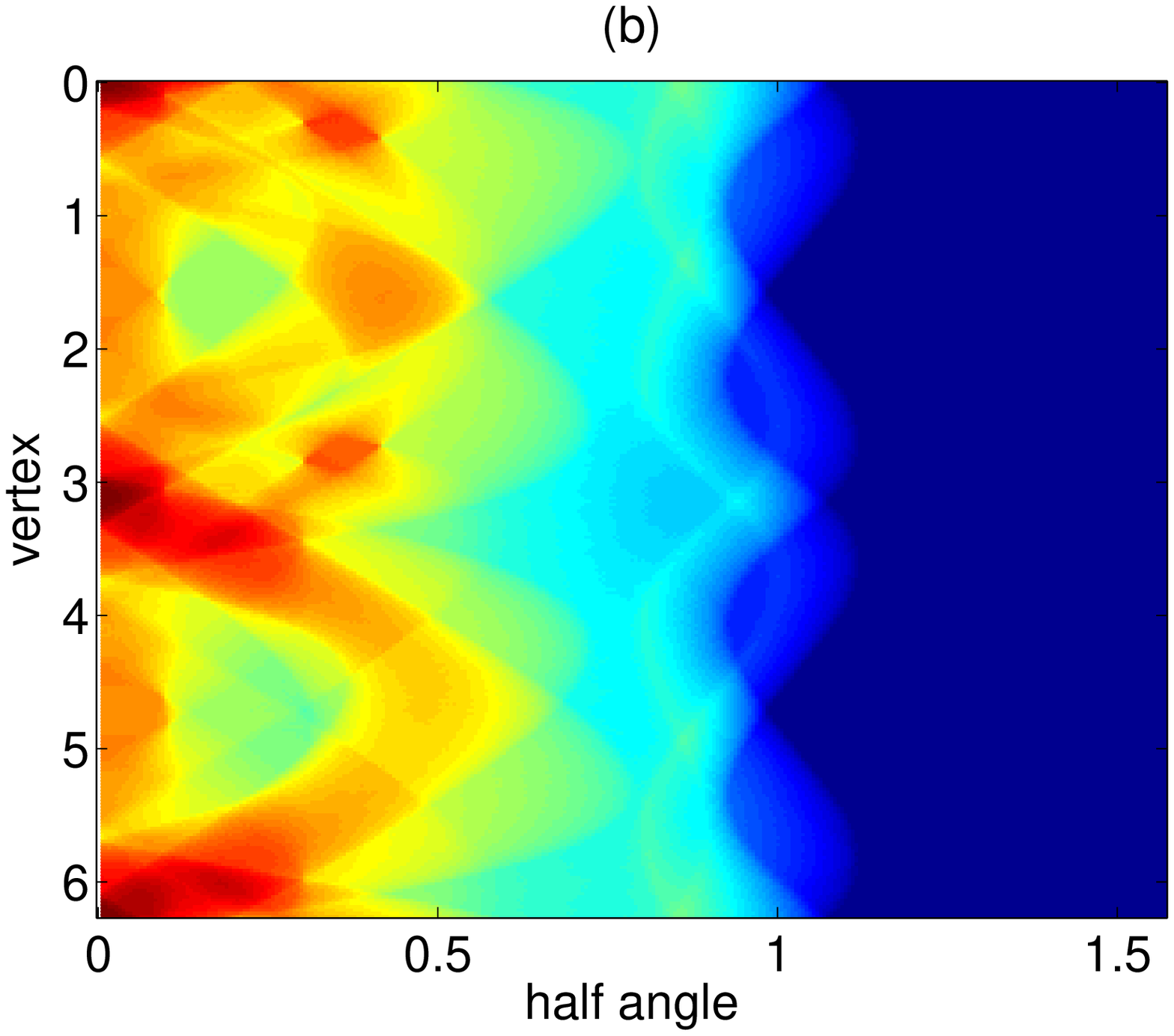}
\includegraphics[width=0.32\textwidth]{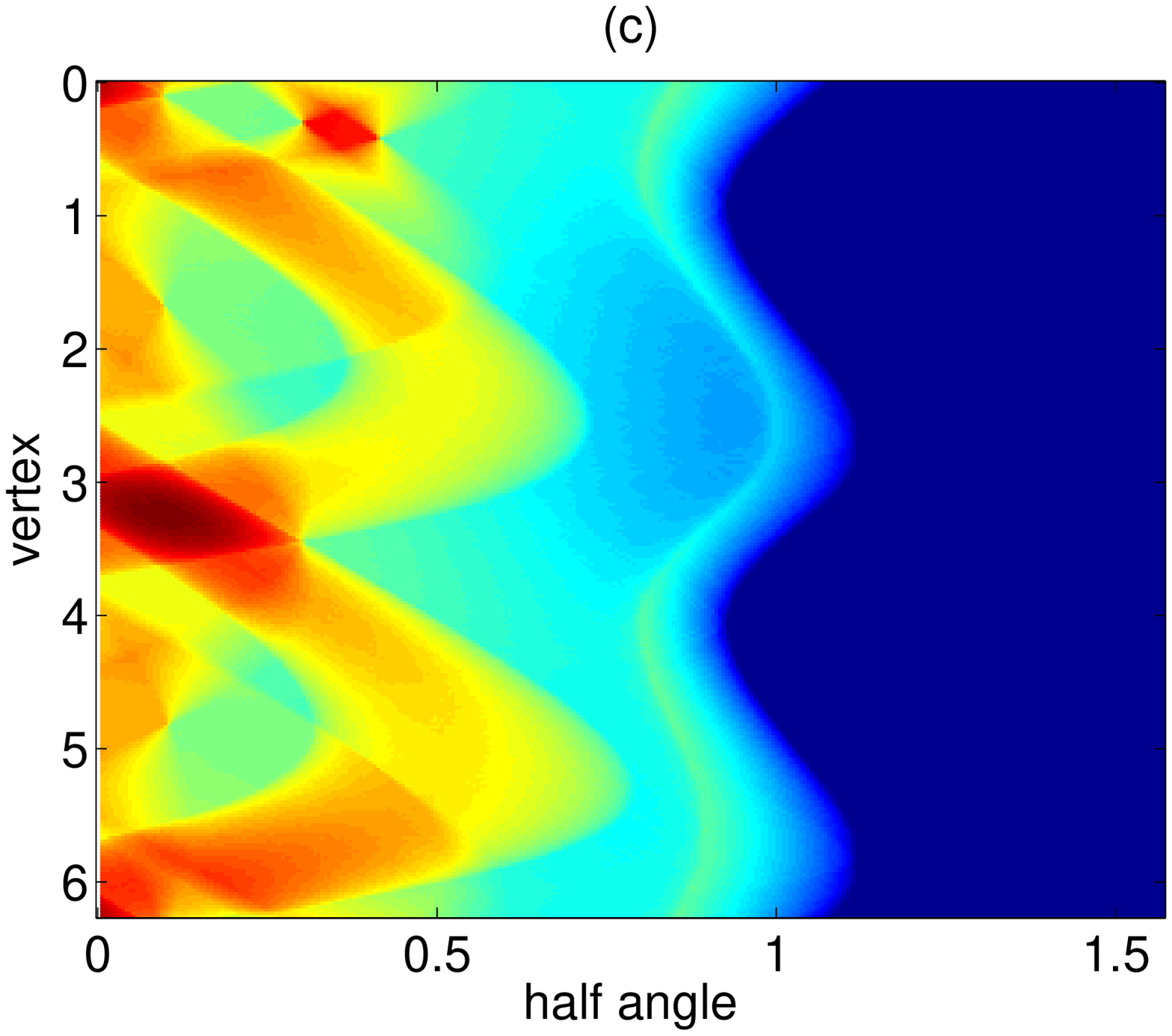}
\caption{ (a) Phantom \label{fig:num} $\F$ used for numerical simulations. (b) Corresponding V-line transform  $\Vo \F$. (c) Corresponding  X-ray transform  $\Xo \F$.}
\end{figure}

Formula \eqref{eq:final}  immediately yields to the following discrete reconstruction algorithm for inverting
the V-line transform using discrete data.

\begin{framed}
\begin{alg}[Discrete reconstruction  algorithm for inverting the  \label{alg:dvrt} V-line transform]\mbox{}
\begin{itemize}[leftmargin=3em]
\item {\bfseries\scshape\ul{Step 1:}}
Use the FFT to compute  the $\GG[n, i] \simeq (\Vo \F)_{n}(s_i)$.

\item {\scshape\bfseries\ul{Step 2:}}
For some  $\epsilon >0$, compute
$\HH[n, i] \simeq  (\Ro \F)_{n}(s_i)$ by evaluating  \eqref{eq:tik}.

\item{\scshape\bfseries\ul{Step 3:}}
Compute $\FF[n, i] \simeq  \FF_n(r_i)$ by evaluating \eqref{eq:final}.

\item{\scshape\bfseries\ul{Step 4:}}
 Approximate $\FF(r_i \cos \varphi_k, r_i \cos \varphi_k)$  by applying the inverse FFT.
\end{itemize}
\end{alg}
\end{framed}

Algorithm~\ref{alg:dvrt} is numerically efficient in the following sense. If $M = \mathcal {O}(N)$, easy arguments show that the proposed algorithm only requires  $\mathcal {O}(N^3)$ floating point operations for reconstructing the phantom at $N^2$ reconstruction points.
This is the same complexity as filtered backprojection reconstruction algorithms have. Notice further that by evaluating~\eqref{eq:final} in Step~3 in
Algorithm~\ref{alg:dvrt} we actually  implement the inversion formula~\eqref{eq:invRb}  for the regular Radon transform  derived by Perry~\cite{Per75}.
The proposed implementation of~\eqref{eq:invRb} is of interest on its own and is different from the implementation given in~\cite{Hansen81circular}. Finally note that a different reconstruction strategy  based on \eqref{eq:tik} (or \eqref{eq:inv-aux}) would be  to first recover $\Ro \F$  by applying the inverse FFT algorithm and then to apply any existing reconstruction algorithm for the Radon transform such as the filtered backprojection algorithm.

\subsection{Numerical examples}

The reconstruction procedure outlined above has been implemented in \textsc{Matlab} and  tested on a  discretized version of a Smiley phantom shown in Figure~\ref{fig:num} (a)  sampled on a Cartesian  $201 \times 201$ grid.
For  implementing the  V-line  transform we
first  numerically computed the X-ray transform
$\Xo \F   (\ph, \psi) = \int_0^\infty \F(\theta(\ph)-r(\cos(\ph-\psi),\sin(\ph-\psi)))  \dr$    by computing the  ray integrals  using the composite trapezoidal rule. We then evaluated the V-line  transform using $\Vo \F   (\ph, \psi) =
\sum_{\sigma = \pm 1} \Xo \F   (\ph, \sigma \psi )$.
Figure~\ref{fig:num}~(b) and  (c) show the numerically computed
V-line and X-ray transforms  for $m = 256$ vertex positions  and $N=201$ opening angles in the interval $[0, \pi/2]$.

\begin{figure}[htb!]
\includegraphics[width=0.48\textwidth]{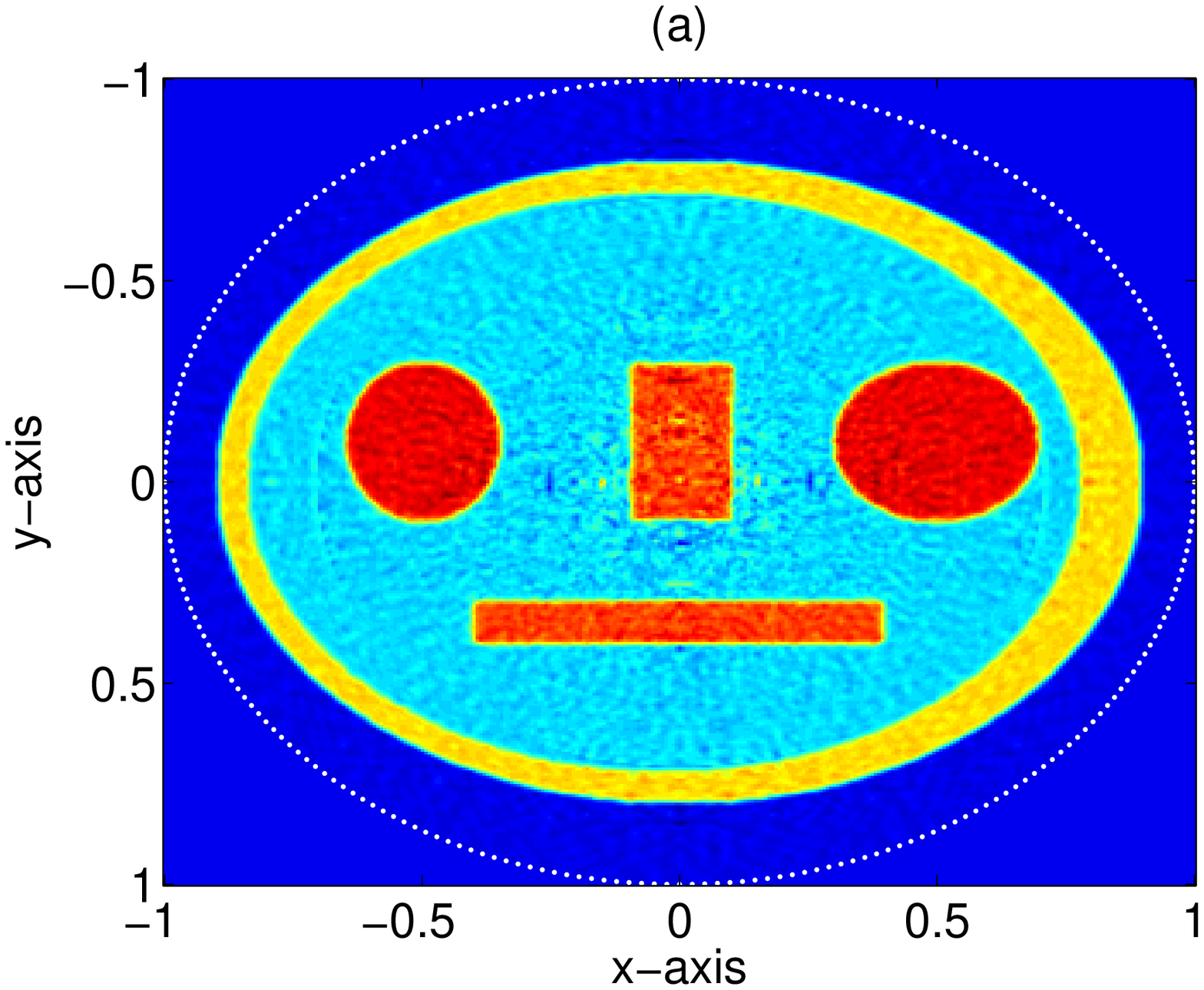}
\includegraphics[width=0.48\textwidth]{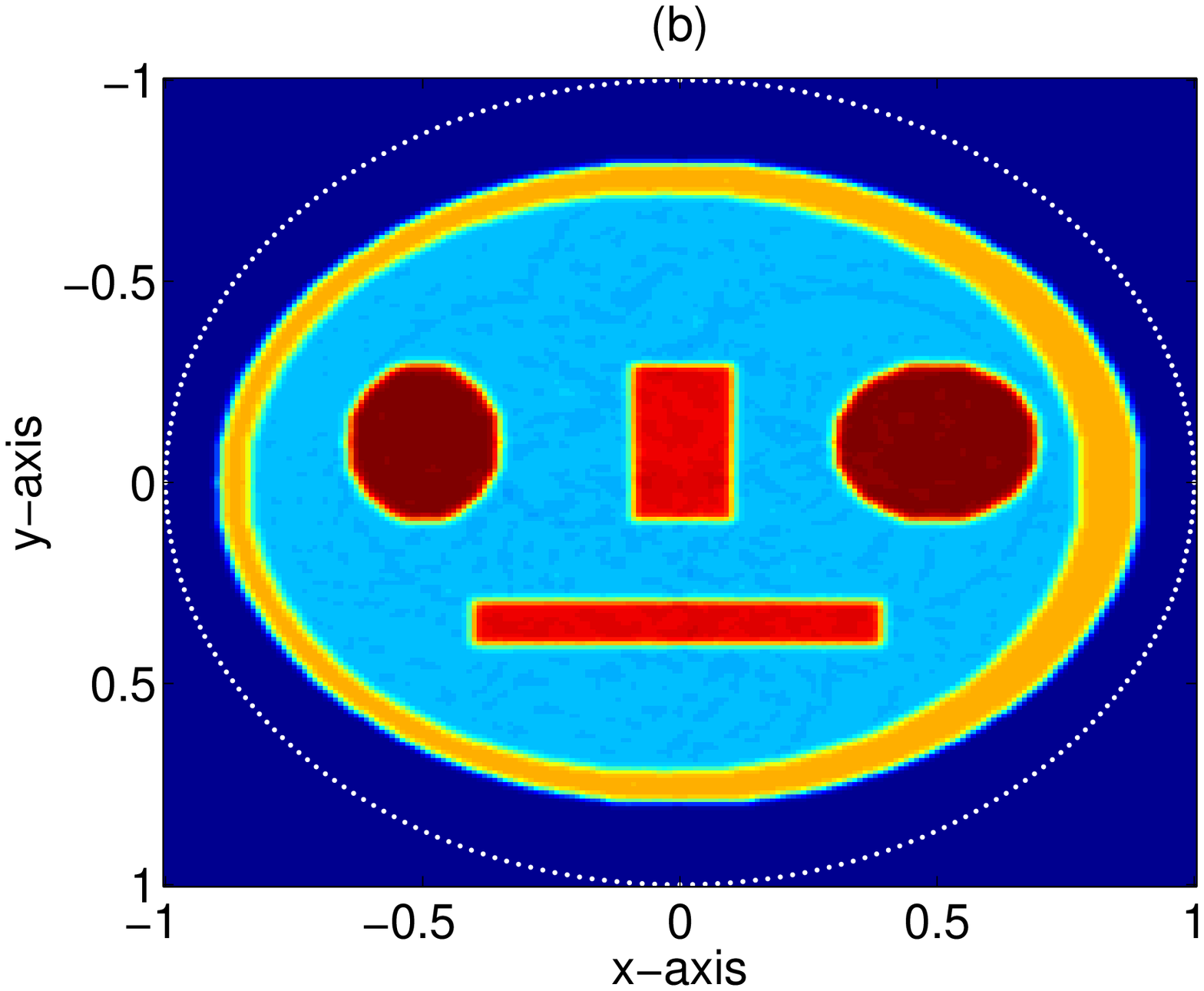}\\
\includegraphics[width=0.48\textwidth]{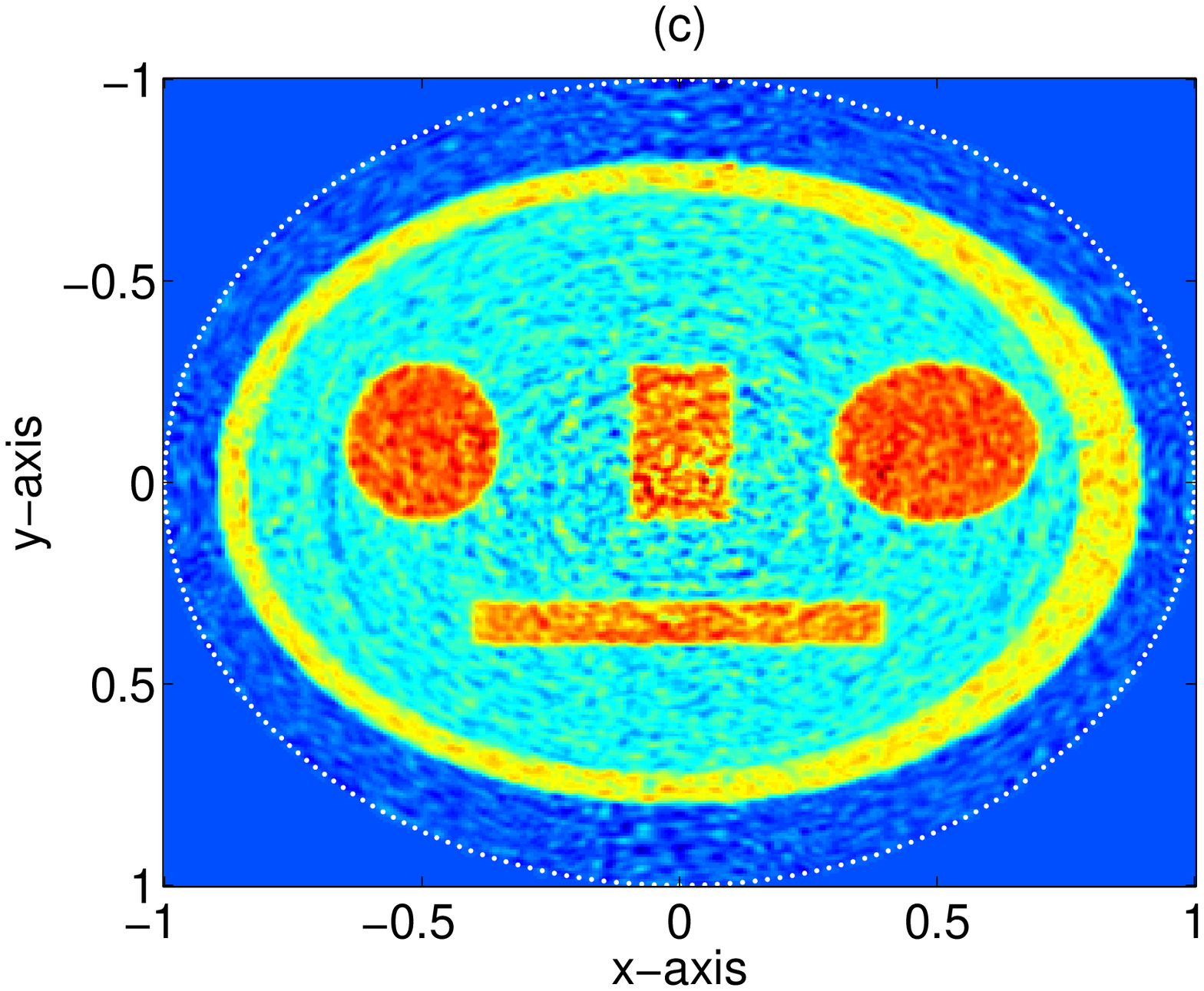}
\includegraphics[width=0.48\textwidth]{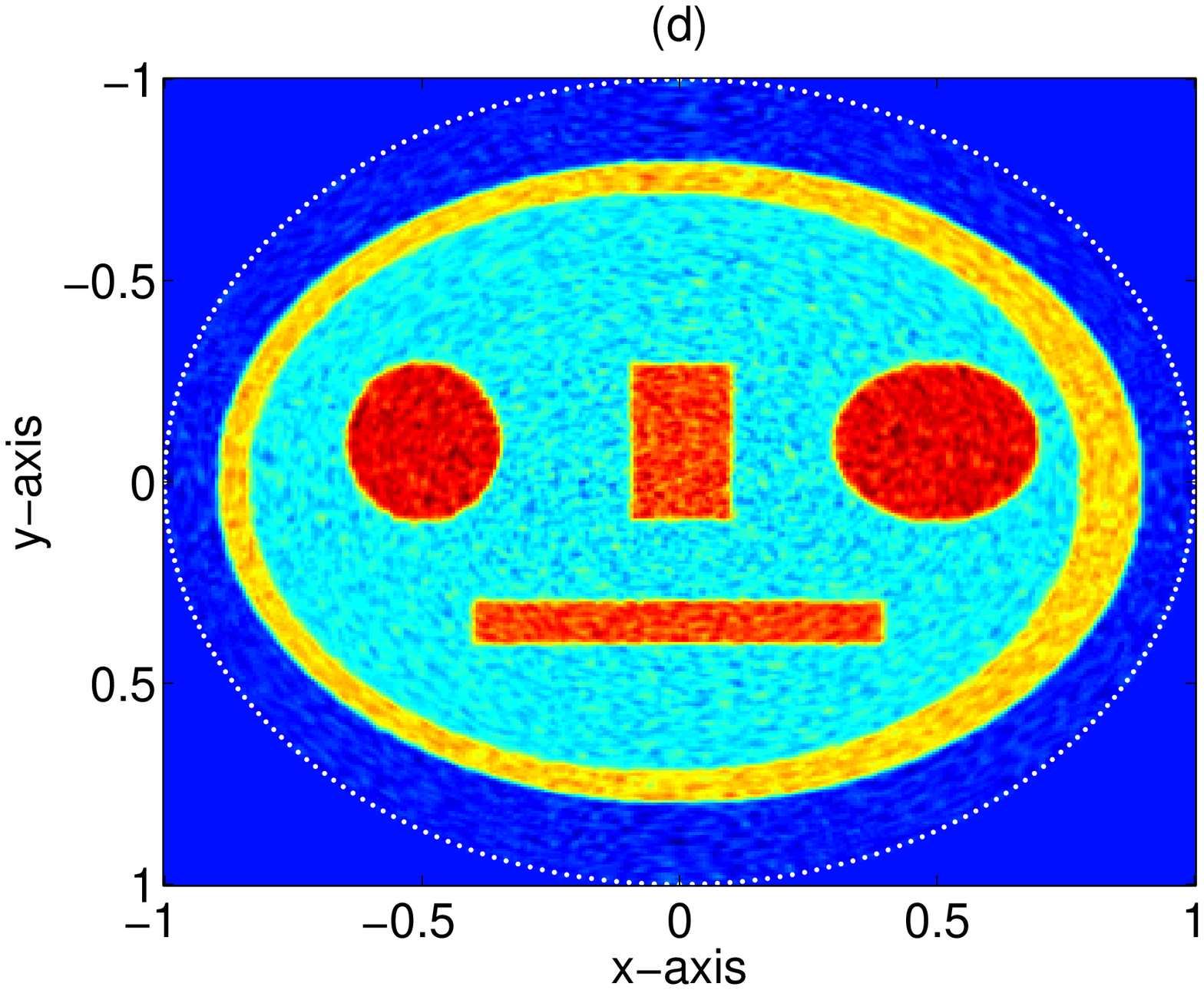}
\caption{(a) Reconstruction  \label{fig:rec} from numerically computed V-line transform. (b) Reconstruction from numerically computed  X-ray transform. (c) Reconstruction  from  V-line transform after adding Gaussian white noise to the data. (d) Reconstruction  from  X-ray transform after adding Gaussian white  noise to the data.}
\end{figure}

 The numerical reconstruction from the V-line transform    using Algorithm~\ref{alg:dvrt} evaluated  on a Cartesian  $201 \times 201$ grid is shown in Figure~\ref{fig:rec}~(a).   The regularization parameter has been taken as $\epsilon = 0.005$. For comparison purpose we also applied our reconstruction algorithm  to the X-ray transform, see Figure~\ref{fig:rec}~(b). For inverting the X-ray transform, equation \eqref{eq:tik} is replaced by
$ \HH[n,j] \coloneqq \GG[n,j]  / \exp(-i n(\arcsin(s_j) - \pi/2) )$.
(In fact, such a reconstruction procedure is justified by~\eqref{eq:VRT} and \eqref{eq:vfnrf}.)
Because $\exp(-i n t)  \neq 0$, no regularization is necessary for inverting $\Xo \F$.
Finally, in order to demonstrate the stability of our algorithm we added  $5 \%$ noise to the data and repeated the computations. For inverting the V-line transform we used an increased  regularization parameter equal to $\epsilon = 0.05$. Again the results for both the V-line and the X-ray transform are quite good. However, the inversion of the V-line transform is slightly more sensitive to the noise which is expected due to the  zeros of the function $s \mapsto \cos(n(\arcsin(s) - \pi/2) )$ appearing in the denominator of the inversion formula for $\Vo \F$.

\section{The  conical Radon transform with vertices on a cylinder}
\label{sec:cone}

In this section  we study the inversion of the conical Radon transform on the cylinder using 4D data. For that purpose, for $ (\ph, z,\beta, \psi) \in [0, 2\pi)  \times \RR \times (0, \pi)^2$, we use the  following notations:
\begin{itemize}
\item $\theta(\ph) \coloneqq (\cos \ph, \sin \ph) \in \sph^1$;
\item  $\axis{\ph}{\beta} \coloneqq (-\theta(\ph)\sin(\beta),\cos(\beta))\in \sph^2$;
\item  $C (\ph, z,\beta, \psi)   \coloneqq (\theta(\ph), z) + \set{x \in \RR^3 \mid
\inner{\axis{\ph}{\beta} }{x} =  \norm{x} \cos \psi }$.
\end{itemize}
The set $C (\ph, z,\beta, \psi) $ is a one-sided circular cone, having a vertex $(\theta(\ph), z) \in \sph^1 \times \RR$, a symmetry axis $\set{r\axis{\ph}{\beta} \mid r >0}$ pointing to the symmetry axis  of the
cylinder $\sph^1 \times \RR$, and a half
opening angle  $\psi \in (0, \pi)$; see Figure~\ref{fig:cylinder}~(a).

\begin{definition}[Conical\label{def:crt} Radon transform on the cylinder]
Let $k \in  \NN_0=\set{0,1, \dotsc} $ and $\f \in C_c^\infty (D_1(0)\times \RR)$.
We define  the \ul{(weighted) conical Radon transform}  of $\f$ by
\begin{align}  \label{eq:conical Radon transform}
\Co_k \f  \colon   [0, 2\pi)  \times \RR  \times (0, \pi)^2  &\to \RR \colon
\\\nonumber
 (\ph, z,\beta, \psi) &\mapsto
 \int_{C (\ph, z,\beta, \psi) } \f(x) \norm{x-(\theta(\ph),z)}^{k-1}\dS(x)  \,.
\end{align}
\end{definition}

The  weighted conical Radon transform maps the function  $\f$ to integrals over  members of the four-dimensional set of all cones $\sset{C (\ph, z,\beta, \psi)\subseteq \RR^3  \mid (\ph, z,\beta, \psi) \in [0, 2\pi)  \times \RR \times (0, \pi)^2 }$.
The parameter $k \in  \NN_0$  allows to include a radial weight that can be adjusted to a particular application at hand. In the literature on Compton camera imaging the
case $k=0$  and $k = 1$ have been used. In~\cite{Smi05}, the case $k=1$ is referred  to as  the  surface integral model,
and the  case $k=0$ as  the cone-beam line integral model.

In the following we present two inversion methods for inverting $\Co_k$.
The first one is  based on  reducing the  conical Radon transform to the V-line transform,  whereas the second one is based on reducing the conical Radon transform  to the Radon transform.

\subsection{Method 1: Reduction to the V-line transform}
\label{sec:CV}

The first inversion method consists in first recovering  the V-line
transform  from $\Co_k \f$  and subsequently recovering
$\f$ by  inverting  the V-line transform. For that purpose we make use of the (weighted) X-ray transform
\begin{equation}
(\Xo_k \f)  \colon [0,2\pi)\times\RR \times \RR^3   \to \RR \colon
(\ph, z,\uu) \mapsto   \int_0^\infty \f((\theta(\ph),z)+r\uu))r^k\dr \,,
\end{equation}
for $k \in  \NN_0$ and $\f \in C_c^\infty(D_1(0) \times \RR)$.
The  X-ray transform  consists of integrals over rays $ \set{(\theta(\ph), z) + r  \uu \mid r >0}$ with a vertex $ (\theta(\ph), z) \in \sph^1  \times \RR$  and a direction $\uu \in \RR^3$.

\begin{lemma}\label{lem:CV}
Let   $f\in C^\infty_0( D_1(0) \times \RR)$ and $(\ph, z) \in [0,2\pi)\times\RR$.
For any $\yy=(y_1,y_2)\in\RR^2$ with $\norm{\yy}\leq1$, we have
\begin{multline}\label{eq:CX}
\sum_{\sigma = \pm 1}(\Xo_k f)\kl{\ph, z,y_1\theta(\ph)+\sigma\sqrt{1-\snorm{\yy}^2} \, \theta(\ph)^\bot,y_2}
\\
=\frac{\sqrt{1-\snorm{\yy}^2}}{2\pi^2}\int_0^\pi \int^\pi_0\frac{  \partial_{\psi}   [\Co_k f(\ph, z,\al,\psi)/\sin\psi]}{\cos(\psi)- \inner{\yy}{(-\sin(\beta),\cos(\beta))}} \rmd  \psi \dbeta \,.
\end{multline}

Further, for any  $\vv \in \sph^1$ we have
\begin{equation}\label{eq:CX2}
(\Xo_0 \f)(\ph, z , \vv, 0) = \frac{(-1)^{k-1}}{(k-1)!} \int_{-\pi/2}^0
\sin^{k-1} (\gamma)(\partial_z^k \Xo_k \f) \kl{\ph, z, \vv\cos(\gamma), \sin(\gamma) }\rmd \gamma.
\end{equation}
\end{lemma}

\begin{proof} Using the one-dimensional delta-distribution, we have
\begin{align}
\Co_k f (\ph, z,\beta, \psi)
&= \nonumber
\sin(\psi)
\int_{\sph^2}\int^\infty_0
f\kl{ (\ph, z) +  r  \om }  \delta \kl{ \inner{\axis{\ph}{\beta} }{\om} - \cos (\psi) } r^k \dr \dS(\om)
\\  \nonumber
&=
\sin(\psi)  \int_{\sph^2}(\Xo_k f)(\ph, z,\om)
\delta \kl{ \inner{\axis{\ph}{\beta} }{\om} - \cos (\psi) } \dS(\om)
\\  \label{eq:CX-aux}
&=
\sum_{\sigma =\pm 1}
\sin(\psi)  \int_{\sph^2_\sigma}(\Xo_k f)(\ph, z,\om)
\delta \kl{ \inner{\axis{\ph}{\beta} }{\om} - \cos (\psi) } \dS(\om)
\,,
\end{align}
where $\sph^2_{\sigma} \coloneqq \sset{ \omega  \in \sph^2 \mid \inner{(\sigma \theta(\ph)^\bot, 0)}{\omega} > 0 }$. Any  element on the half sphere $\sph^2_\sigma$ can uniquely be written in the
form $\om =  \skl{y_1\theta(\ph) + \sigma \sqrt{1-\snorm{\yy}^2}\theta(\ph)^\bot,y_2}$
for $\yy = (y_1, y_2) \in \RR^2$ with $\norm{\yy} < 1$.
Using this representation we have $\inner{\axis{\ph}{\beta}}{ \om}= \inner{(-\sin(\beta),\cos(\beta))}{\yy}$.
Together with \eqref{eq:CX-aux} and the transformation rule this shows
\begin{multline*}
\Co_k f(\ph, z,\beta,\psi)
= \sum_{\sigma =\pm 1}
\int_{\norm{\yy} < 1} (\Xo_k f) \kl{\ph, z,y_1\theta(\ph)+ \sigma \sqrt{1-\snorm{\yy}^2}\theta(\ph)^\bot,y_2} \\ \times
\delta \kl{\inner{(-\sin(\beta),\cos(\beta))}{\yy}-\cos(\psi)}
\frac{\sin(\psi) \dy}{\sqrt{1-\snorm{\yy}^2}}  \,.
\end{multline*}
Now for fixed $\ph$, $z$ and $k$, we define functions $F_\sigma \colon \RR^2 \to \RR$
by
\begin{equation*}
F_\sigma (\yy)  \coloneqq
\begin{cases}
(\Xo_k f) \kl{\ph, z,y_1\theta(\ph)+\sqrt{1-\snorm{\yy}^2}\theta(\ph)^\bot,y_2} \dfrac{1}{\sqrt{1-\snorm{\yy}^2}}  & \text{ for } \norm{\yy} \leq 1 \\
0 & \text{ for } \norm{\yy} > 1 \,,
\end{cases}
\end{equation*}
and recall that the 2-dimensional Radon transform can be written in the form   $ \Ro \F (\al, s) = \int_{\RR^2} F(\yy)\delta \kl{\inner{(\cos(\al),\sin(\al))}{\yy}-s}\dy$; we obtain
$\Co_k f(\ph, z, \beta,\psi)/\sin\psi = \sum_{\sigma = \pm 1}
  (\Ro \F_\sigma)(\beta + \pi/2, \cos(\psi))$.
With  the
 formula
$ F(\yy) = \frac{1}{2\pi^2}
	\int_0^\pi
	\int_{\RR}  \skl{\inner{(\cos(\al),\sin(\al))}{\yy}-s}^{-1}(\partial_s \Ro \F)(\al,s) \ds \rmd\al$ for inverting the Radon transform  (see, for example, \cite{Hel99,Nat01})
this yields
\begin{multline*}
 \sum_{\sigma = \pm 1} F_\sigma (\yy)
=
\frac{1}{2\pi^2}\int^{\pi}_0\int^1_{-1}\frac{\partial_{\cos(\psi)}  [\Co_k f(\ph, z,\al-\pi/2,\psi)/\sin\psi]}{\inner{\yy}{(\cos(\al),\sin(\al))}-\cos(\psi)} \rmd(\cos \psi )\rmd \al
\\
=\frac{1}{2\pi^2}\int_0^{\pi}\int^\pi_0 \frac{\partial_{\psi}  [\Co_k f(\ph, z,\al-\pi/2,\psi)/\sin\psi]}{\cos(\psi)-\inner{(\cos(\al),\sin(\al))}{\yy}} \rm\rmd{\psi} \rmd\al.
\end{multline*}
Inserting  the definition of  $F_\sigma $ yields~\eqref{eq:CX}.

By the chain rule we have
\begin{multline*}
	(\partial_\ww^k \Xo_0 \f)(\ph, z,\vv, \ww)
	=
	\int_0^\infty  \partial_\ww^k \f\kl{ (\theta(\ph), z)+r(\vv,\ww)}  \dr
	\\
	=
	 \int_0^\infty  \partial_z^k \f\kl{ (\theta(\ph), z)+r(\vv,\ww)} r^k  \dr
	 =  (\partial_z^k \Xo_k \f) (\ph, z,\vv, \ww)\,.
\end{multline*}
Together with Cauchy's  formula for repeated integration
we obtain
\begin{multline*}
	(\Xo_0 \f)(\ph, z,\vv, \ww)
	= \frac{1}{(k-1)!}
	\int_{-\infty}^\ww
	(\ww - s)^{k-1}
	(\partial_z^k \Xo_k \f) (\ph, z,\vv, s)
	\ds
	\\
	=
	\frac{1}{(k-1)!}
	\int_{-\infty}^\ww
	(\ww-s)^{k-1}
	(\partial_z^k \Xo_k \f)
	\kl{\ph, z, \frac{(\vv, s)}{\sqrt{\snorm{\vv}^2+ s^2}}} \frac{\ds }{\skl{\snorm{\vv}^2+ s^2}^{(k+1)/2}} \,,
\end{multline*}
where for the second equality we used the identity
$(\Xo_k f)(\ph,z,\lambda \uu) = \lambda^{-k-1}(\Xo_k f)(\ph,z,\uu)$ that holds for every  $\lambda>0$ and $\uu \in \RR^3$.

Now we take $\ww=0$ and $\norm{\vv} =1$ in the last displayed
equation,  and make the   substitution $ s =  \tan (\gamma)$. Then $\snorm{\vv}^2+ s^2 = 1/\cos^2(\gamma)$  and $\ds = \rmd \gamma /\cos^2(\gamma)$.
Consequently,
\begin{multline*}
	(\Xo_0 \f) (\ph, z,\vv, 0)
	= \frac{1}{(k-1)!}
	\int_{-\pi/2}^0
	(-\tan  (\gamma))^{k-1}
	(\partial_z^k \Xo_k \f)
	\kl{\ph, z, \vv \cos(\gamma), \sin(\gamma)}
	\\ \times
	\cos^{k-1}(\gamma)
	\rmd\gamma  \quad \text{ for } \vv \in \sph^1 \,,
\end{multline*}
which after simple manipulation yields~\eqref{eq:CX2}.
\end{proof}

The following theorem shows how to explicitly reduce the weighted conical Radon transform to the (two-dimensional) V-line
transform applied in horizontal planes.

\begin{theorem}[Reduction\label{thm:cone} of the conical Radon transform to the V-line transform]
Let $\f \in C_c^\infty \kl{ D_1(0) \times \RR}$.
Then, we have for  $(\ph, z, \psi) \in [0, 2\pi) \times \RR \times (0, \pi/2)$,
\begin{multline}•
\sum_{\sigma = \pm 1}
(\Xo_0 \f) (\ph, z, -\cos(\psi) \theta(\ph) + \sigma\sin(\psi) \theta(\ph)^\bot, 0  )
\\
=    \int_0^\pi \int^\pi_0
\partial_z^k\partial_{\psi} [ \Co_k \f(\ph, z,\beta,\psi)/\sin\psi]
H_k(\beta,\psi)\rmd  \psi \dbeta
\,,
\end{multline}
with
\begin{equation} \label{eq:kernel}
H_k(\beta,\psi) \coloneqq
\frac{(-1)^{k-1}}{2\pi^2 (k-1)!}
\int_{-\pi/2}^0 \frac{   \sin^{k-1} (\gamma)\cos(\gamma)\sin(\psi) \rmd \gamma}{\cos(\psi)- \inner{( \cos(\gamma)\cos(\psi) ,\sin(\gamma))}{(\sin(\beta),\cos(\beta))}}  \,.
\end{equation}
\end{theorem}

\begin{proof}
By Lemma~\ref{lem:CV} we have
\begin{multline}
\sum_{\sigma = \pm 1}
(\Xo_0 \f) (\ph, z, -\cos(\psi) \theta(\ph) + \sigma\sin(\psi) \theta(\ph)^\bot, 0  )
\\
\begin{aligned}
&=
\frac{(-1)^{k-1}}{(k-1)!}
\sum_{\sigma = \pm 1} \int_{-\pi/2}^0  \sin^{k-1} (\gamma)
\\ & \hspace{0.1\textwidth}
\times
(\partial_z^k \Xo_k \f)\kl{\ph, z , -\cos(\gamma)\cos(\psi) \theta(\ph) + \sigma \cos(\gamma)\sin(\psi) \theta(\ph)^\bot , \sin(\gamma)}  \rmd \gamma
\\&=
\frac{(-1)^{k-1}}{(k-1)!} \partial_z^k  \int_{-\pi/2}^0   \sin^{k-1} (\gamma)
\\ & \hspace{0.1\textwidth}
\times \sum_{\sigma = \pm 1} (\Xo_k \f)\kl{\ph, z , -\cos(\gamma)\cos(\psi) \theta(\ph) + \sigma \cos(\gamma)\sin(\psi) \theta(\ph)^\bot , \sin(\gamma)} \rmd \gamma
\\&
=   \frac{(-1)^{k-1}}{2\pi^2 (k-1)!}\int_{-\pi/2}^0\int_0^\pi \int^\pi_0 \frac{  \sin^{k-1} (\gamma)\cos(\gamma)\sin(\psi)\partial_z^k\partial_{\psi} [ \Co_k \f(\ph, z,\beta,\psi)/\sin\psi]\rmd  \psi \dbeta\rmd \gamma}{\cos(\psi)- \inner{( -\cos(\gamma)\cos(\psi) ,\sin(\gamma))}{(-\sin(\beta),\cos(\beta))}}
\,.
\end{aligned}
\end{multline}
After interchanging the order of integrations and using the definition of the kernel $H_k(\beta,\psi)$ yields the desired identity.
\end{proof}

Let $\f \in C_c^\infty \kl{ D_1(0) \times \RR}$.
For any given $z \in \RR$  define $\f_z  \coloneqq \f(\edot, z)$.
Therefore,  the V-line transform of $\f_z$ satisfies
\begin{align} \nonumber
 \kl{ \Vo \f_z}(\ph,\psi)
 &=
 \sum_{\sigma = \pm 1}\int_0^\infty \f(\theta(\ph)-r(\cos(\ph-\sigma\psi),\sin(\ph-\sigma\psi)),z)  \dr
 \\&=  \nonumber
 \sum_{\sigma = \pm 1}
(\Xo_0 \f) (\ph, z, -\cos(\psi) \theta(\ph) + \sigma\sin(\psi) \theta(\ph)^\bot, 0  )
\\&= \label{eq:XV}
    \int_0^\pi \int^\pi_0
\partial_z^k\partial_{\psi} [ \Co_k \f(\ph, z,\beta,\psi)/\sin\psi]
H_k(\beta,\psi, \gamma)\rmd  \psi \dbeta  \,.
\end{align}
Here the second  equality follows from the identity
$-\cos(\psi) \theta(\ph) + \sigma\sin(\psi) \theta(\ph)^\bot
= -(\cos(\ph-\sigma\psi),\sin(\ph-\sigma\psi))$ and the last
equality from Theorem~\ref{thm:cone}. Equation \eqref{eq:XV}
 shows how to recover the V-line transform $\Vo \f_z$ from the conical Radon transform. By subsequently inverting the  V-line transform  one recovers $\f$.
In summary, we conclude the following inversion method for the
conical Radon transform:

\begin{framed}
\begin{alg}[Inversion of the \label{alg:crt1} conical Radon transform, Method 1]\mbox{}
\begin{itemize}[leftmargin=3em]
\item {\bfseries\scshape\ul{Step 1:}}
Compute the kernel $H_k$ defined in \eqref{eq:kernel}.

\item {\scshape\bfseries\ul{Step 2:}}
 Evaluate \eqref{eq:XV} to recover $\Vo \f_z$ from $\Co_k \f$
 for every $z \in \RR$.

\item{\scshape\bfseries\ul{Step 3:}}
 For every $z \in \RR$ recover $\f_{z}$ from $\Vo \f_z$ by means of  Algorithm~\ref{alg:vrt}.
 \end{itemize}
\end{alg}
\end{framed}

Theorem~\ref{thm:cone} also  implies that the considered conical
Radon transform $f \mapsto \Co_k \f$ is uniquely invertible and can be reconstructed by Algorithm~\ref{alg:crt1}. How to efficiently  implement  Algorithm~\ref{alg:crt1}  is a topic of  future work.

\subsection{Method 2: Reduction to the  Radon Transform}
\label{sec:CR}

The second method is probably simpler than the method presented above, but
only works for the cases $k = 0,1$. It  based on a relation between the conical Radon transform and the regular 3-dimensional Radon transform,
that has first been first derived  by Smith~\cite{Smi05}.
For that purpose we use the following additional notation:
\begin{itemize}
\item
$(\Ro\f) (\omega,s) \coloneqq \int_{\omega^\bot}f(s\omega + \yy) \dS(\yy)$ for  $ (\om,s) \in \sph^2 \times \RR$ denotes the regular 3-dimensional  Radon transform of  $f \in C_c^\infty(\RR^3)$;

\item
$(\Ho g)(\om, s) \coloneqq  \frac{1}{\pi}
\int_{\RR}g(\om, t) \rmd t /(s-t)$ for  $ (\om,s) \in \sph^2 \times \RR$
denotes  the Hilbert transform  in the second component of a function $g \in C_c^\infty(\sph^2 \times \RR)$.
\end{itemize}

\begin{lemma}[Reduction\label{lem:CRk} of  the conical Radon transform to the Radon transform] Let $\f \in C_c^\infty(\RR^3)$.
Then, for $(\ph, z, \beta) \in [0, 2\pi) \times \RR \times (0, \pi) $, we have
\begin{align}
	 \label{eq:CR0}
	(\Ho\partial_s \Ro\f)(\axis{\ph}{\beta}, z \cos(\beta) - \sin(\beta))
	&=
	-\frac{1}{\pi}
	\int^\pi_0  (\Co_0\f)(\ph,z, \beta,\psi)
	\frac{\rmd{\psi}}{\cos (\psi)^2} \\
	\label{eq:CR1}
	(\Ho \Ro\f)(\axis{\ph}{\beta}, z \cos(\beta) - \sin(\beta))
	&=
	-\frac{1}{\pi}
	\int^\pi_0  (\Co_1\f)(\ph,z, \beta,\psi)
	\frac{\rmd{\psi}}{\cos (\psi)}
	 \,.
\end{align}
\end{lemma}

\begin{proof}
Expressing  $\Co_k$ for $k=0,1$ in terms of the
one-dimensional delta-distribution and performing several coordinate substitutions yield
\begin{multline*}
\int^\pi_0 \Co_k \f(\ph, z,\beta,\psi)\cos^{k-2} (\psi)\rmd{\psi}
\\
\begin{aligned}
&=
\int^\pi_0
\int_{\sph^2}\int^\infty_0
f\kl{ (\theta(\ph), z) +  r  \om }  \delta \kl{ \inner{\axis{\ph}{\beta} }{\om} - \cos (\psi) } r^k \dr \dS(\om)
\cos^{k-2}(\psi)\rmd{(\cos(\psi))}
\\
&=
\int_{\sph^2}\int^\infty_0 \f\kl{(\theta(\ph),z)+r\om } (\inner{\axis{\ph}{\beta} }{ \om})^{k-2}r^{k }\dr \dS(\om)
\\&=
\int_{\RR^3} \f\kl{(\theta(\ph),z)+\xx }(\inner{\axis{\ph}{\beta} }{ x})^{k-2}  \rmd{\xx}
\\
&=
\int_{\RR^3}   \f(\xx )(\inner{\axis{\ph}{\beta}}{\xx}-\inner{\axis{\ph}{\beta}}{(\theta(\ph),z)} )^{{k-2}}\rmd{\xx}
\\
&=
\int_{\RR}\int_{\axis{\ph}{\beta}^\bot}  f(s \axis{\ph}{\beta}+\yy )(s+\sin(\beta)-z\cos(\beta))^{{k-2}}\rmd{\yy}\rmd{s}
\\
&=
\int_{\RR}  (\Ro\f)(\axis{\ph}{\beta},s) (s+\sin(\beta)-z\cos(\beta))^{k-2}\rmd{s}.
\end{aligned}
\end{multline*}
Here the third equality follows after introducing  spherical coordinates
$\xx \gets r\om$, the fourth equality follows after the change of variables
$\xx \gets (\theta(\ph),z)+\xx$,  the fifth  equality  follows after the substitution  $\xx \gets s \axis{\ph}{\beta} + \yy$, and the last equality follows form the definition of the Radon transform.  Now, the definition of the Hilbert transform  and performing one integration by parts in the case $k=0$ yields  \eqref{eq:CR0}, \eqref{eq:CR1}.
\end{proof}

Using Lemma~\ref{lem:CRk} we can recover $ \Ho\partial_s^{1-k} \Ro \f$  from the conical Radon transform. By applying the inverse Hilbert  transform and the inverse Radon transform afterwards one then recovers the original  function. This yields the following reconstruction method.
\begin{framed}
\begin{alg}[Inversion of the \label{alg:crtk2} conical Radon transform for $k=0,1$, Method 2]\mbox{}
\begin{itemize}[leftmargin=3em]
\item {\bfseries\scshape\ul{Step 1:}}
Recover $ \Ho\partial_s^{1-k} \Ro \f$ from $\Co_k\f$ by evaluating~\eqref{eq:CR0} or \eqref{eq:CR1}.

\item {\scshape\bfseries\ul{Step 2:}}
Recover $\partial_s^{1-k}  \Ro\f$ from $\Ho \partial_s^{1-k} \Ro \f$ by applying  inverse Hilbert  transform.

\item{\scshape\bfseries\ul{Step 3:}}
Recover $\f$ from $\partial_s^{1-k} \Ro  \f$ by  applying the inverse Radon transform.
 \end{itemize}
\end{alg}
\end{framed}
While Algorithm~\ref{alg:crtk2} can efficiently be implemented, one can combine three steps to obtain an explicit inversion formula. Such inversion formulas are also useful for theoretical investigations. By using the standard  filtered backprojection  type
inversion formulas  for the Radon transform (see, for example, \cite{Hel99,Nat01})
\begin{align}
f(\xx) &= \label{eq:invR11}
-\frac{1}{8 \pi^2}
\Delta_{\xx}\int_{\sph^2} \Ro\f(\om,\inner{\om}{\xx})\dS(\om)
\\
f(\xx)&= \label{eq:invR21}
-\frac{1}{8 \pi^2}\int_{\sph^2}\partial_t^2 \Ro\f(\om,\inner{\om}{\xx})\dS(\om),
\end{align}
Lemma~\ref{lem:CRk} yields the  following result.

\begin{theorem}For
$f\in C^\infty(D_1(0) \times \RR)$ we have
\begin{align} \label{eq:invCR12}
\f(\xx)
& =-\frac{1}{8\pi^4}\int_{\sph^2}\int_{\RR}\int^\pi_0\frac{ (\partial_z^{2}\Co_k \f)(\ph, z,\beta,\psi) \cos^{-k}(\beta)\rmd{\psi} \rmd{z} \dS(\axis{\ph}{\beta})}{(\inner{\xx}{\axis{\ph}{\beta}}-z\cos(\beta)+\sin(\beta))\cos^{2-k }(\psi)}
&&\text{for } k\in \set{0,1}
\\ \label{eq:invCR11}
 \f(\xx)
 &=
 -\frac{1}{8\pi^4}\Delta_{\xx}\int_{\sph^2}\int_{\RR}\int^\pi_0\frac{ \Co_1 \f(\ph, z,\beta,\psi)\cos^{}(\beta) \rmd{\psi} \rmd{z} \dS(\axis{\ph}{\beta})}{(\inner{\xx}{\axis{\ph}{\beta}}-z\cos(\beta)+\sin(\beta))\cos(\psi)}
 &&\text{for } k=0  \,.
 \end{align}
\end{theorem}

\begin{proof}
%Suppose $k=1$.
Since $\Ho\Ho g=-g$, we have
\begin{align*}
\partial_s^{1-k}\Ro\f(\axis{\ph}{\beta},s)
&=
-\partial_s^{1-k}(\Ho\Ho\Ro\f)(\axis{\ph}{\beta},s)
\\&=
\frac{(-1)^k}{\pi}\int_{\RR} \frac{\Ho\Ro\f(\axis{\ph}{\beta},z\cos(\beta)-\sin(\beta))\cos(\beta) \rmd{z}}{(s-z\cos(\beta)+\sin(\beta))^{2-k}}
\\&=
-\frac{1}{\pi}\int_{\RR} \frac{\partial_z^{1-k}\Ho\Ro\f(\axis{\ph}{\beta},z\cos(\beta)-\sin(\beta))\cos^k(\beta) \rmd{z}}{s-z\cos(\beta)+\sin(\beta)}
\\&=
\frac{1}{\pi^2}\int_{\RR}\int^\pi_0\frac{ \partial_z^{1-k}\Co_k  f(\ph, z,\beta,\psi)\cos^{}(\beta) \rmd{\psi} \rmd{z}}{(s-z\cos(\beta)+\sin(\beta))\cos^{2-k}(\psi)} \,.
\end{align*}
Together with \eqref{eq:invR11} this yields  \eqref{eq:invCR11} for $k=1$.
Further, integration by parts shows
\begin{multline*}
(\partial_s^2\Ro\f)(\axis{\ph}{\beta},t) \\
\begin{aligned}
& =
(\partial_s^{1+k}\partial_s^{1-k}\Ro\f)(\axis{\ph}{\beta},t)
\\ &=
\frac{(-1)^{k+1}(k+1)!}{\pi^2}\int_{\RR}\int^\pi_0\frac{ (\partial_z^{1-k}\Co_k \f)(\ph, z,\beta,\psi)\cos(\beta) \rmd{\psi} \rmd{z}}{(s-z\cos(\beta)+\sin(\beta))^{2+k}\cos^{2-k}(\psi)}
\\ &=
\frac{(-1)^{k+1}}{\pi^2}\int_{\RR}\int^\pi_0\partial^{1+k}_z\left[\frac1{s-z\cos(\beta)+\sin(\beta)}\right](\partial_z^{1-k}\Co_k \f)(\ph, z,\beta,\psi)\cos^{-k}(\beta) \frac{ \rmd{\psi} \rmd{z}}{\cos^{2-k}(\psi)}
\\ &=
\frac{1}{\pi^2}\int_{\RR}\int^\pi_0\frac{(\partial_z^{2}\Co_k \f)(\ph, z,\beta,\psi)\cos^{-k}(\beta)}{(s-z\cos(\beta)+\sin(\beta))^{}\cos^{2-k}(\psi)}  \rmd{\psi} \rmd{z}\,,
\end{aligned}
\end{multline*}
which together with   \eqref{eq:invR21}  yields \eqref{eq:invCR12}.
\end{proof}

The inversion formulas \eqref{eq:invCR12} for $k=0$ and \eqref{eq:invCR11} for $k=1$  can further be rewritten in terms of the conical back-projection operator $\Co^{\sharp}_k  g$
 that is defined by
\begin{multline} \label{eq:cbpk}
\Co^{\sharp}_k  g(\xx)=\int^\pi_0\int^{\pi}_0\int_{\RR} \int^{2\pi}_0g(\ph, z,\beta,\psi)\snorm{\xx-(\theta(\ph),z)}^{k-1}\\ \times \delta(\inner{\axis{\ph}{\beta}}{\xx}-z\cos(\beta)+\sin(\beta)-\snorm{\xx-(\theta(\ph),z)}\cos(\psi))\sin(\psi) \rmd{\ph} \rmd{z} \rmd{\beta} \rmd{\psi}\,.
\end{multline}
for $g\in C^\infty ([0,2\pi)\times\RR\times(0,\pi)^2)$ and
 $\xx\in\RR^3$.
As shown in Appendix~\ref{ap:Cstar}, the  operator $\Co^\sharp_k$ is the (formal) $L^2$-adjoint of $\Co_k$.

\begin{corollary}\label{cor:invCR}
Let $f\in C^\infty(D_1(0) \times \RR)$.
Then we have for $\xx \in \RR^3$,
\begin{align}\label{eq:invCR3}
\f(\xx)
&=\frac{(-1)^{k}}{8\pi^4}\Delta_{\xx}^{k}\Co^{\sharp}_k \left[ \int^\pi_0\frac{\partial_z^{1-k}\Co_k \f(\ph, z,\beta,\psibar)\cos^{  }(\beta) \sin(\beta) \rmd{\psibar}}{\cos^{2-k}(\psi)\cos^{2-k }(\psibar)}\right](\xx).
%\f(\xx)
%&=-\frac{1}{8\pi^4}\Co^{0,\sharp} \left[ \int^\pi_0\frac{\partial_z\Co_0 \f(\ph, z,\beta,\psibar) \sin(\beta) \rmd{\psibar}}{\cos(\psi)\cos(\psibar)}\right](\xx) \,.
\end{align}
\end{corollary}

\begin{proof}
See Appendix~\ref{ap:invCR}
\end{proof}

Using Corollary~\ref{cor:invCR}, we can easily derive a stability estimate for the conical Radon transform. For that purpose we denote by $\Fo\f$  the Fourier transform of $\f$ and set
\begin{align*}
\snorm{\f}_{-1}^2
&=
\int_{\RR^3}|\Fo f(\xi)|^2(\norm{\xi}^2+1)^{-1}\rmd{\xi},\qquad
\\
\norm{\Co_k \f}^2
&=\int^{\pi}_0\int_{\RR}\int^{2\pi}_0 \int^\pi_0 \left|\frac{\Co_k f(\ph, z,\beta,\psi)}{\cos^{2-k}(\psi)}\right|^2\cos(\beta) \sin(\beta) \rmd{\psi}\rmd{\ph} \rmd{z} \rmd{\beta}
\\
\norm{\Co_k \f}^2_1
&=\norm{\Co_k \f}^2+\norm{\partial_z\Co_k \f}^2.
\end{align*}

\begin{theorem}\label{prop:norm}
For  $\f \in C_c^\infty(D_1(0) \times \RR)$
we have $ \snorm{\f}_{-1}\leq \snorm{\Co_1 f}$ and $ \snorm{\f}_{}\leq \pi/2\snorm{\Co_0 f}_1$.
\end{theorem}

\begin{proof}
By Corollary~\ref{cor:invCR} with $k=1$ we have
\begin{align}
\lVert &\f \rVert_{-1}^2
\leq
\int_{\RR^3}|\Fo\f(\xi)|^2\norm{\xi}^{-2}\rmd{\xi}\label{eq:normf1}
\\
&=-(2\pi)^3 \int_{\RR^3}f(\xx)\Delta_\xx^{-1} f(\xx)\rmd{\xx}\nonumber\\
&=\frac1\pi\int_{\RR^3} f(\xx) \Co^\sharp_1 \left[ \int^\pi_0\frac{\Co_1 f(\ph, z,\beta,\psibar)\cos(\beta) \sin(\beta) \rmd{\psibar}}{\cos(\psi)\cos(\psibar)}\right](\xx) \rmd{\xx}\nonumber\\
&=\frac1\pi\int^\pi_0\int^{\pi}_0\int_{\RR} \int^{2\pi}_0 \Co_1 f(\ph, z,\beta,\psi) \int^\pi_0\frac{\Co_1 f(\ph, z,\beta,\psibar)\cos(\beta) \sin(\beta) \rmd{\psibar}}{\cos(\psi)\cos(\psibar)} \rmd{\ph} \rmd{z} \rmd{\beta} d \psi\nonumber\\
&=\frac1\pi\int^{\pi}_0\int_{\RR}\int^{2\pi}_0 \left(\int^\pi_0 \frac{\Co_1 f(\ph, z,\beta,\psi)}{\cos(\psi)}\rmd{\psi}\right)^2\cos(\beta) \sin(\beta) \rmd{\ph} \rmd{z} \rmd{\beta}\nonumber,
\end{align}
where in the  second last equality holds because $\Co^\sharp_1 $ is the  $L^2$-adjoint of $\Co_1$.
Applying Jensen's inequality to the inner integral completes our proof.

Now suppoe $k=0$. Similar to \eqref{eq:normf1} we have
\begin{align*}
\lvert \f \rVert^2&=\frac1\pi\int^{\pi}_0\int_{\RR}\int^{2\pi}_0
\kl{\int^\pi_0 \frac{\Co_0 f(\ph, z,\beta,\psi)}{\cos^2(\psi)}\rmd{\psi} }
\\ &{}\hspace{0.1\textwidth} \times
\kl{\int^\pi_0 \frac{\partial_z \Co_0 f(\ph, z,\beta,\psi)}{\cos^2(\psi)}\rmd{\psi}}\cos(\beta) \sin(\beta) \rmd{\ph} \rmd{z} \rmd{\beta}\\
&\leq\frac1\pi\left(\int^{\pi}_0\int_{\RR}\int^{2\pi}_0 \left(\int^\pi_0 \frac{\Co_0 f(\ph, z,\beta,\psi)}{\cos^2(\psi)}\rmd{\psi}\right)^2\cos(\beta) \sin(\beta) \rmd{\ph} \rmd{z} \rmd{\beta}\right)^{1/2}
\\ &{}\hspace{0.1\textwidth} \times
\left(\int^{\pi}_0\int_{\RR}\int^{2\pi}_0 \left(\int^\pi_0 \frac{\partial_z \Co_0 f(\ph, z,\beta,\psi)}{\cos^2(\psi)}\rmd{\psi}\right)^2\cos(\beta) \sin(\beta) \rmd{\ph} \rmd{z} \rmd{\beta}\right)^{1/2},
\end{align*}
where in the second line, we used the Cauchy-Schwarz inequality.
By Jensen's inequality,
\begin{align*}
\snorm{\f}^2
&\leq \pi\left(\int^{\pi}_0\int_{\RR}\int^{2\pi}_0 \int^\pi_0 \left|\frac{\Co_0 f(\ph, z,\beta,\psi)}{\cos^2(\psi)}\right|^2\rmd{\psi}\cos(\beta) \sin(\beta) \rmd{\ph} \rmd{z} \rmd{\beta}\right)^{1/2}\\
&{}\hspace{0.1\textwidth} \times \left(\int^{\pi}_0\int_{\RR}\int^{2\pi}_0 \int^\pi_0 \left|\frac{\partial_z \Co_0 f(\ph, z,\beta,\psi)}{\cos^2(\psi)}\right|^2\rmd{\psi}\cos(\beta) \sin(\beta) \rmd{\ph} \rmd{z} \rmd{\beta}\right)^{1/2}\\
&\leq \frac{\pi}2\left(\int^{\pi}_0\int_{\RR}\int^{2\pi}_0 \int^\pi_0 \left|\frac{\Co_0 f(\ph, z,\beta,\psi)}{\cos^2(\psi)}\right|^2\rmd{\psi}\cos(\beta) \sin(\beta) \rmd{\ph} \rmd{z} \rmd{\beta}\right)^{}\\
&{}\hspace{0.1\textwidth} + \frac{\pi}2\left(\int^{\pi}_0\int_{\RR}\int^{2\pi}_0 \int^\pi_0 \left|\frac{\partial_z \Co_0 f(\ph, z,\beta,\psi)}{\cos^2(\psi)}\right|^2\rmd{\psi}\cos(\beta) \sin(\beta) \rmd{\ph} \rmd{z} \rmd{\beta}\right)\,.\end{align*}
\end{proof}

\section{Conclusion}
\label{sec:discussion}

In this paper we studied the weighted conical Radon $\Co_k$  with vertices on the cylinder  and presented  two explicit reconstruction procedures (see
Algorithms~\ref{alg:crt1} and~\ref{alg:crtk2}).  The first approach is based on reducing $\Co_k$ to the V-line transform with vertices on the circle. For the V-line transform we derived an explicit  inversion formula based on Fourier  series expansion and a corresponding efficient discrete reconstruction algorithm. We believe that also Algorithms~\ref{alg:crt1} and~\ref{alg:crtk2} for inverting  $\Co_k$ can  be implemented efficiently. Future  work will be done to  numerically implement  these reconstruction methods. We  intend to compare these methods  with iterative procedures in terms of computation time and image quality for realistically simulated data.

\appendix

\section{Proofs}
\label{app:proofs}

\subsection{Formal $L^2$-adjoint of $\Co_k$}
\label{ap:Cstar}

We have
\begin{align*}
&\int^\pi_0\int^{\pi}_0\int_{\RR} \int^{2\pi}_0 \Co_k f(\ph, z,\beta,\psi)g(\ph, z,\beta,\psi)\rmd{\ph} \rmd{z} \rmd{\beta} \rmd{\psi}
\\
&= \int^\pi_0\int^{\pi}_0\int_{\RR} \int^{2\pi}_0  \int_{\RR^3}  f((\theta(\ph),z)+\xx) \delta(\inner{\axis{\ph}{\beta}}{\xx}-\snorm{\xx}\cos(\psi))
\\ &
\hspace{0.05\textwidth} \times  g(\ph, z,\beta,\psi) \frac{\sin(\psi)\rmd{\xx}}{\snorm{\xx}^{1-k}} \rmd{\ph} \rmd{z} \rmd{\beta} \rmd{\psi}
\\
&=  \int_{\RR^3} \int^\pi_0\int^{\pi}_0\int_{\RR} \int^{2\pi}_0 f(\xx) g(\ph, z,\beta,\psi)\snorm{\xx-(\theta(\ph),z)}^{k-1}
\\
&  \hspace{0.05\textwidth}
\times \delta(\inner{\axis{\ph}{\beta}}{\xx}-z\cos(\beta)+\sin(\beta)-\snorm{\xx-(\theta(\ph),z)}\cos(\psi)) \sin(\psi)\rmd{\ph} \rmd{z} \rmd{\beta} \rmd{\psi} \rmd{\xx}
\\&=
\int_{\RR^3}f(\xx)\Co^{\sharp}_k g(\xx)\rmd{\xx} \,.
\end{align*}
Here for the  first and second equalities, we made use of the change  of variables $\xx \gets r\om$ and $\xx \gets (\theta(\ph),z)+\xx$, respectively.

\subsection{Proof of Corollary~\ref{cor:invCR}}
\label{ap:invCR}

It is enough to show that for a fixed $\xx\in\RR^3$ and $k=0,1$, we have
\begin{multline}\label{eq:mainofcor}
\Co^{\sharp}_k \left[ \int^\pi_0\frac{\partial_z^{1-k}g(\ph, z,\beta,\psibar)\cos^{  }(\beta)\sin(\beta) \rmd{\psibar}}{\cos^{2-k }(\psi)\cos^{2-k }(\psibar)}\right](\xx)\\
=(-1)^{k+1}\int_{\sph^2}\int_{\RR}\int^\pi_0\frac{ \partial_z^{2-2k}g(\ph, z,\beta,\psi)\cos^{ k}(\beta) \rmd{\psi} \rmd{z} \dS(\axis{\ph}{\beta})}{(\inner{\axis{\ph}{\beta}}{\xx}-z\cos(\beta)+\sin(\beta))\cos^{2-k}(\psi)}.
\end{multline}
%By a routine computation, we can show it.
By the definition of $\Co^\sharp_k$ and the homogeneity of the one-dimensional delta-distribution, the left hand side of (\ref{eq:mainofcor}) becomes
\begin{align*}
&\int^{\pi}_0\int_{\RR} \int^{2\pi}_0\int^\pi_0\int^\pi_0\frac{\partial_z^{1-k}g(\ph, z,\beta,\psibar)\cos^{ }(\beta)\sin(\beta) }{\cos^{2-k}(\psi)\cos^{2-k }(\psibar)\snorm{\xx-(\theta(\ph),z)}^{2-k}}
\\
& \hspace{0.1\textwidth}
\times \delta\left(\frac{\inner{\axis{\ph}{\beta}}{\xx}-z\cos(\beta)+\sin(\beta)}{\snorm{\xx-(\theta(\ph),z)}}-\cos(\psi)\right)\sin(\psi)\rmd{\psibar}\rmd{\psi} \rmd{\ph} \rmd{z} \rmd{\beta}
\\&=
\int^{\pi}_0\int_{\RR} \int^{2\pi}_0\int^\pi_0\frac{ \partial_z^{1-k}g(\ph, z,\beta,\psibar)\cos^{ }(\beta) \sin(\beta)\rmd{\psibar}\rmd{\ph} \rmd{z} \rmd{\beta}}{(\inner{\axis{\ph}{\beta}}{\xx}-z\cos(\beta)+\sin(\beta))^{2-k}\cos^{2-k }(\psibar)}  \,.
\end{align*}
Now the surface measure of the sphere is $\sin(\beta)\rmd{\ph}  \rmd{\beta}$ when $\ph$ and $\beta$ are the azimuthal and polar angles, respectively.
Thus we have
\begin{multline*}
\Co^{\sharp}_k \left[ \int^\pi_0\frac{\partial_z^{1-k}g(\ph, z,\beta,\psibar)\cos^{ }(\beta)\sin(\beta) \rmd{\psibar}}{\cos^{2-k }(\psi)\cos^{2-k }(\psibar)}\right](\xx)
\\=\int_{\sph^2}\int_{\RR} \int^\pi_0\frac{ \partial_z^{1-k}g(\ph, z,\beta,\psi)\cos^{ }(\beta) \rmd{\psi} \rmd{z}\rmd{S(\ph,\beta)} }{(\inner{\axis{\ph}{\beta}}{\xx}-z\cos(\beta)+\sin(\beta))^{2-k}\cos^{2-k }(\psi)} \,.
\end{multline*}
If $k=1$, the proof is done.
If $k=0$, then the integration by parts completes the proof.

\section{Generalization to higher dimension}
\label{app:hd}

In this section we generalize the results for the conical Radon transform presented in  Subsection~\ref{sec:CR} to a general dimension.
For the following let $n \geq 3$.

\begin{definition}[The conical Radon transform on the cylinder in $\RR^n$]
Let $\f \in C_c^\infty (B_1(0)\times \RR)$. We define the  \ul{conical  Radon transform in $\RR^n$} $\Co \f  \colon   \sph^{n-2}  \times \RR  \times (0, \pi)^2  \to \RR$ by
\begin{equation*}
\Co f(\theta,z,\beta,\psi) \coloneqq \sin(\psi)\int_{\sph^{n-1}}\int^\infty_0 f((\theta,z)+r\om )r^{n-2} \delta(\om\cdot\axis{\theta}{\beta}-\cos(\psi)) \dr \dS(\om)\,,
\end{equation*}
where $\axis{\theta}{\beta} \coloneqq (-\theta\sin(\beta),\cos(\beta))$. \end{definition}
%We furthers set $\Co_{n-2}\f=\Co \f$.

As in the three-dimensional case the
formal $L^2$-adjoint of $\Co$ is given by
\begin{multline}
\Co^\sharp g(\xx) = \int^\pi_0\int^{\pi}_0\int_{\RR} \int_{\sph^{n-2}}g(\theta,z,\beta,\psi)\\
\times \delta(\xx\cdot\axis{\theta}{\beta}-z\cos(\beta)+\sin(\beta)-\snorm{\xx-(\theta,z)}\cos(\psi))\sin(\psi) \dS(\theta) \rmd{z} \rmd{\beta} \rmd{\psi} \,.
\end{multline}

 We further use the following notations:
 \begin{itemize}
 \item $\Ro\f(\om, s)$ for the regular $n$-dimensional Radon transform;
 \item  $\Fo\f (\xi)$ for the $n$-dimensional Fourier transform;
 \item $(-\Delta_{\xx})^{(n-1)/2} \f \coloneqq \Fo^{-1} \skl{  \snorm{\xi}^{n-1} \Fo \f}$  for the fractional  Laplacian;
 \item $\snorm{\f}_{-(n-1)/2}^2 \coloneqq \int_{\RR^n}|\Fo\f(\xi)|^2(\norm{\xi}^2+1)^{-(n-1)/2}\rmd{\xi}$;
\item  $\snorm{\Co f}^2 \coloneqq \int^{\pi}_0\int_{\RR}\int_{\sph^{n-2}}\int^\pi_0  \left|{\Co f(\theta,z,\beta,\psi)}/\cos(\psi)\right|^2\cos(\beta) \sin(\beta) \, \rmd{\psi} \dS(\theta) \rmd{z} \rmd{\beta} $.
 \end{itemize}

 Similar to the three-dimensional case one then has the following
 results.

\begin{theorem}
For any $\f \in C_c^\infty ( B_1(0) \times \RR)$ the following hold:

\noindent\textsc{Relation to Radon transform:} For every $(\theta, \beta, t)
\in \sph^{n-1} \times (0, \pi) \times \RR$,
\begin{equation*}
\Ro\f(\axis{\theta}{\beta},t)=\frac1{\pi^2}\int_{\RR}\int^\pi_0\frac{ \Co f(\theta,z,\beta,\psi)\cos(\beta) \rmd{\psi} \rmd{z}}{(t-z\cos(\beta)+\sin(\beta))\cos(\psi)} \,.
\end{equation*}

\noindent\textsc{Inversion formulas:} For every $\xx \in \RR^n$,
\begin{align}
\f(\xx) &=\frac{1}{\pi(2\pi)^n}\Delta^{(n-1)/2}_{\xx}\int^\pi_0 \int_{\sph^{n-2}}\int_{\RR}\int^\pi_0\frac{ \Co f(\theta,z,\beta,\psi)\cos(\beta)  \sin(\beta) \rmd{\psi} \rmd{z} \dS({\theta})\rmd \beta }{(\xx\cdot\axis{\theta}{\beta}-z\cos(\beta)+\sin(\beta)) \cos(\psi)} \,,
\\
\f(\xx)
&=\frac{1}{\pi(2\pi)^n}\Delta^{(n-1)/2}_{\xx}\Co^\sharp \left[\int_0^\pi\frac{ \Co f(\theta,z,\beta,\psibar)\cos(\beta)\sin(\beta) \rmd{\psibar}}{\cos(\psi)\cos(\psibar)}\right](\xx) \,.
\end{align}

\noindent\textsc{Stability estimate:}  $\norm{\f}_{-(n-1)/2}\leq \norm{\Co f}$.
\end{theorem}

\begin{proof}
The proof is analogous to the three-dimensional case and is therefore
omitted.
\end{proof}

Similar to the three-dimensional case one  can derive inversion formulas and  stability estimates for the weighted conical Radon transform.

\end{document}